\UseRawInputEncoding
\documentclass[a4paper,oneside,11pt]{article}
\usepackage{amsmath}
\usepackage{amssymb}
\usepackage{titlesec}
\usepackage{graphicx}
\usepackage{fancyhdr}
\usepackage{latexsym}
\usepackage{mathrsfs}
\usepackage{booktabs}
\usepackage{mathrsfs}
\usepackage{mathtools}
\usepackage{amsthm}
\usepackage{wasysym}
\usepackage{cite}
\usepackage{color}
\usepackage{amsfonts}
\usepackage{pgf,pgfarrows,pgfnodes,pgfautomata,pgfheaps}
\usepackage{amsmath,amssymb}
\usepackage{graphics}
\usepackage{multimedia}
\usepackage{graphics}
\usepackage{cite}
\usepackage{latexsym}
\usepackage{amsmath}
\usepackage{amssymb}
\usepackage[dvips]{epsfig}
\usepackage{amscd}
\numberwithin{equation}{section}
\newtheorem{definition}{Definition}[section]
\newtheorem{theorem}{Theorem}[section]
\newtheorem{lemma}{Lemma}[section]
\newtheorem{remark}{Remark}[section]
\newtheorem{proposition}{Proposition}[section]

\newcommand{\ba}{\begin{aligned}}
\newcommand{\ea}{\end{aligned}}
\newcommand{\be}{\begin{equation}}
\newcommand{\ee}{\end{equation}}

\newcommand{\bnn}{\begin{eqnarray*}}
\newcommand{\enn}{\end{eqnarray*}}

\newcommand{\thatsall}{\hfill$\Box$}
\date{}

\title{Global Existence   of  Strong and Weak Solutions to 2D Compressible Navier-Stokes System   in  Bounded Domains with Large Data and Vacuum}

 \author{Xinyu F{\small AN}$^{a}$, Jiaxu   L{\small I}$^{ a} $,  Jing L{\small I}$^{ b,c ,a} $ \thanks{Email addresses:  fanxinyu17@mails.ucas.edu.cn (X. Y. Fan), Jiaxvlee@gmail.com (J. X. Li), ajingli@gmail.com  (J. Li)  }  \\  {\normalsize a.  School of Mathematical Sciences,}\\
{\normalsize  University of Chinese Academy of Sciences, Beijing 100049, P. R. China;}\\
{\normalsize b. Department of Mathematics, }\\ {\normalsize  Nanchang University, Nanchang 330031, P. R. China;} \\ {\normalsize c. Institute of Applied Mathematics, AMSS,} \\ {\normalsize \&   Hua Loo-Keng Key Laboratory of Mathematics,}\\
{\normalsize  Chinese Academy of Sciences,    Beijing 100190,
P. R. China }}
\begin{document}
\maketitle
\begin{abstract}We study the barotropic compressible Navier-Stokes system where the shear viscosity is a positive constant and the bulk one   proportional to a power of the density with the power bigger than one and a third.  The system is subject to the Navier-slip boundary conditions in  a general two-dimensional  bounded simply connected      domain. For initial density allowed to vanish, we establish   the global existence of strong and weak solutions without any restrictions on   the size of initial value. To get over the difficulties brought by boundary,  on the one hand, we    apply Riemann mapping theorem and  the pull-back Green's function method to get a pointwise representation of the effective viscous flux.  On the other hand, observing that the
orthogonality   is preserved under conformal mapping due to its  preservation on the angle, we use the slip boundary conditions to reduce the integral
representation to the desired commutator form whose singularities can be   cancelled  out by using the estimates on the spatial gradient of the velocity.  
\end{abstract}
\section{Introduction and main results}
We study the  barotropic compressible Navier-Stokes system in a two-dimensional (2D) domain $\Omega$:
\begin{equation}
\left\{ \begin{array}{l}
  \rho_t+\mathrm{div}(\rho u)=0, \\
  (\rho u)_t+\mathrm{div}(\rho u\otimes u)+\nabla P=\mu\triangle u+\nabla((\mu+\lambda)\mathrm{div} u),   \label{11}
       \end{array} \right. \\
\end{equation}
where $\rho=\rho(x, t)$ and $u=(u_1(x, t), u_2(x, t))$ represent the unknown density and velocity respectively,  and
the pressure $P$ is given by
\begin{equation}
P=a\rho^\gamma,  a>0,  \gamma>1.
\end{equation}
We also have the following hypothesis on the shear viscosity coefficients $\mu$ and the
bulk one $\lambda$:
\begin{equation}\label{qiu09}
0<\mu=\mathrm{constant}, \ \lambda(\rho)=b\rho^\beta,
\end{equation} with positive constants $b$ and $\beta.$
We set $a=b=1$ without loss of generality.  In this paper,  we assume that $\Omega$ is a simply connected bounded $C^{2,1}$-domain in $\mathbb{R}^2.$   In addition,  the system is   subject to the given initial data
\begin{equation}
\rho(x, 0)=\rho_0(x), \ \rho u(x, 0)=m_0(x), \ x\in\Omega,
\end{equation}
and Navier slip boundary conditions:
\begin{equation}
u\cdot n=0\ \mathrm{and}\ \mathrm{curl} u=0\ \mathrm{on}\ \partial\Omega,  \label{15}
\end{equation} where $n=(n_1,n_2)$ denotes the unit outer normal vector of the boundary $\partial \Omega.$

There is large number of literature about the strong solvability for multidimensional compressible Navier-Stokes system with constant viscosity coefficients.  The history of the area may trace back to Nash \cite{1962Le} and Serrin \cite{ser1},  who established the local existence and uniqueness of classical solutions respectively for the density away from vacuum. The first result of global classical solutions was due to Matsumura-Nishida \cite{1980The}   for initial data close to a non-vacuum equilibrium in $H^{3}.$   Hoff \cite{1995Global,hoff2005} then studied the problem with discontinuous initial data, and introduced a new type of a priori estimates   on the material derivative  $\dot{u}$. The major breakthrough in the frame of weak solutions was due to Lions \cite{1998Mathematical}, where he successfully obtained the global existence of weak solutions just under the assumption that the energy is finite initially. For technical reasons, the exponent  $\gamma$ was larger than $\frac{9}{5}$, which was further released to the critical case $\gamma>\frac{3}{2}$ by Feireisl et al\cite{feireisl2004dynamics}. Recently,  Huang-Li-Xin \cite{hlx21} and Li-Xin \cite{lx01} established the global existence and uniqueness of classical solutions, merely assumed the initial energy small enough, where large oscillations were available. No restriction about the support of initial density was attached, even compact support is allowed. More recently, for the Navier-slip boundary conditions in  general bounded domains, Cai-Li \cite{caili01}  obtain the global existence and exponential growth of classical solutions with vacuum provided that the initial energy is suitably small.

In contrast, positive results without limitation on the size of initial value are rather fewer. Vaigant-Kazhikov   \cite{vaigant1995} pioneered in large initial value theory. They obtained a unique global strong solution under the restriction $\beta>3$ in rectangle domain. We note that the viscous coefficients depending on density seems crucial in large value theory. 
Very recently, Huang-Li \cite{huang2016existence,hl21} applied some new ideas based on commutator theory and blow up criterion, and improved the conclusion in periodic case, and even for the Cauchy problem in the whole space (see \cite{jwx1} also), demanding only $\beta>\frac{4}{3}$. Up to now, $\beta>\frac{4}{3}$ still seems to be  the best result one may expect. However,  for general domains, the theory of large initial data is still blank, boundary terms do bring some essential difficulties. Therefore the aim of the paper is to study the global existence of strong and weak solutions with large initial data in general simply connected domains.

Before stating the main results,  we explain the notations and conventions used throughout this paper.
For a positive integer $k$ and $1 \leq p < \infty$,  the standard $L^p$ spaces   and Sobolev ones are denoted as follows:
\begin{equation*}\begin{cases}L^p=L^p(\Omega), \,\, W^{k, p}=W^{k, p}(\Omega), \,\, H^k=W^{k, 2}(\Omega) ,\\
\|f\|_{L^p}=\|f\|_{L^p(\Omega)},   \, \|f\|_{W^{k, p}}=\|f\|_{W^{k, p}(\Omega)}, \, \|f\|_{H^k}=\|f\|_{W^{k, 2}(\Omega)}.\end{cases}
\end{equation*}

The material derivative and the transpose gradient are given by
\begin{equation}\label{ou1r}
\frac{\mathrm{D}}{Dt}f=\dot{f}\triangleq\frac{\partial}{\partial t}f+u\cdot\nabla f, \ \ \
\nabla^{\bot}\triangleq(\partial_2, -\partial_1).
\end{equation}

First we define weak and strong solution as follows.
\begin{definition}Let $T > 0$ be a finite constant.  A solution $(\rho,  u)$ to  \eqref{11}  is called a weak solution if they satisfy \eqref{11} in the sense of distribution. Moreover, when all the derivatives involved in \eqref{11} are regular distributions, and \eqref{11} hold almost everywhere in $\Omega\times (0,  T)$, we call the solution a strong one.
\end{definition}

With our definition we state the main result concerning the global existence of strong solutions  as follows:
\begin{theorem}\label{thmq1}
Let $\Omega$ be a simply connected bounded domain in $\mathbb{R}^2$ with $C^{2,1}$ boundary $\partial\Omega$.  Assume that
\begin{equation}
\beta> \frac{4}{3}, \ \gamma> 1,  \label{17}
\end{equation}
and that the initial data $(\rho_0\ge 0,  m_0)$ satisfy  for some $q > 2$,
\begin{equation}
 \rho_0\in W^{1, q}(\Omega), \ u_0 \in \{H^1(\Omega)|u_0\cdot n=0, {\rm curl}u_0=0 \mbox{  on }\partial \Omega\}, \, m_0 = \rho_0u_0.  \label{18}
\end{equation}
Then the problem \eqref{11}--\eqref{15} has a unique strong solution $(\rho,  u)$ in $\Omega\times(0,  \infty)$ satisfying for any $0<T<\infty, $
\begin{equation}
 \begin{cases}
\rho\in C([0, T]; W^{1, q} ), \, \rho_t\in L^\infty( 0, T ;L^2 ), \\
u\in L^\infty(0, T;H^1 )\cap L^{1+1/q}(0, T;W^{2, q} ), \\
\sqrt tu\in L^\infty(0, T;H^{2} )\cap L^2(0, T;W^{2, q} ), \, \sqrt tu_t\in L^2(0, T;H^1 ), \\  \label{19}
\rho u\in C([0, T] ;L^2), \, \sqrt{\rho}u_t\in L^2((0, T)\times\Omega).
       \end{cases}
\end{equation}

\end{theorem}

The second result concerns the global existence of weak solutions.
\begin{theorem}\label{thmq2} Under the conditions of Theorem \ref{thmq1} except for  $\rho_0\in W^{1,q}(\Omega)$  in \eqref{18}   being replaced by $\rho_0\in L^\infty(\Omega).$  Then, there exists at least one weak solution $(\rho,u)$ of the problem \eqref{11}-\eqref{15}  in
$\Omega\times(0,  \infty)$ satisfying for any $0<\tau\leq T<\infty $ and $p\geq 1$
\begin{equation}\label{111}
 \begin{cases}
\rho\in L^{\infty}(0, T; L^{\infty})\cap C([0,T];L^p),\\
u\in L^\infty(0, T;H^1),u_t\in L^2(\tau,T;L^2),\nabla u\in L^\infty(\tau,T;L^p),\\  {\rm curl}u, (2\mu+\lambda)\mathrm{div}u-P \in L^2( 0,T;H^1  )\cap L^{3/2}( 0,T;L^\infty  ).
       \end{cases}
\end{equation}
 Moreover, there exists a positive constant $C$ depending on  $ \Omega, T,\mu,\beta,\gamma,\|\rho_0\|_{L^{\infty}}$, and $\|u_0\|_{H^1}$ such that \be\label{puw1}  \inf_{(x,t)\in \Omega\times(0,T)}\rho(x,t)\ge C^{-1}\inf_{x\in \Omega}\rho_0(x).\ee
\end{theorem}

\begin{remark} Compared with the previous  known results on the global existence of strong solutions with large data   \cite{vaigant1995,jwx,jwx1,hl21,huang2016existence}, our   Theorem  \ref{thmq1}   seems  to be the first concerning the global existence of strong   solutions to the compressible Navier-Stokes system  in general two-dimensional  bounded domains with large data.
\end{remark}

\begin{remark} It should be mentioned here that  \eqref{puw1} implies  that    the weak solutions obtained by Theorem \ref{thmq2}  will not
exhibit vacuum states in any finite time  provided that no
vacuum states are present initially.  However, for  Lions-Feireisl's  weak solutions  \cite{1998Mathematical,feireisl2004dynamics}, whether this phenomenon still holds or not remains open.
\end{remark}

\begin{remark} For technical reason, we still need assume that $\beta>4/3$ which is the same as that of \cite{hl21,huang2016existence,jwx1}. However,   it seems that $\beta>1$ is the extremal case for the system   \eqref{11}--\eqref{qiu09}  (see \cite{vaigant1995} or Lemma \ref{004} below). Therefore, it would be interesting to study the case of $1<\beta\le 4/3$ which   is left for the future.

\end{remark}


We now make some comments about the analysis of the whole paper.  Similar to \cite{hl21,huang2016existence,jwx1},   the key issue of the existence of global solutions is to derive the upper bound of density $\rho$.
One possible way is to   rewrite the conservation of mass formally in following way
\begin{equation}\label{qu10}
\frac{\mathrm{D}}{dt}(\theta(\rho)+\Delta^{-1}\mathrm{div}(\rho u))+ P =u\cdot\nabla\Delta^{-1}\mathrm{div}(\rho u) -\Delta^{-1}\mathrm{div} \mathrm{div}(\rho u\otimes u)
\end{equation}
where \be \label{qis1}\theta(\rho)\triangleq 2\mu \log \rho+\beta^{-1}\rho^\beta,\ee and $\Delta^{-1}$ is taken in the proper sense. A remarkable fact is that the combination of the   two terms on the righthand side  of \eqref{qu10}   forms a  Calderon-type commutator in some cases, which indeed improves the integrability. For example, the theory applies perfectly in the case of whole space $\mathbb{R}^{2}$ (\cite{hl21}) and that of periodic boundary $\mathbb{T}^{2}$ (\cite{huang2016existence}) due to the fact that one can  interchange the $\Delta^{-1}$ and $\nabla$ at both cases. But for general bounded domains, the interchangeability fails,  and hence, the classical commutator theory is
no longer available, which is indeed a major difficulty in our situation.

To overcome it,  we first take a look at the mechanism of commutator for the case of  $\mathbb{T}^{2}$ and $\mathbb{R}^{2}.$  Rewriting  commutator in singular integral form as
\begin{equation*}
\left[u, R_iR_j\right](\rho u)=\int\frac{u(x)-u(y)}{|x-y|^{2}}\rho u(y)dy,
\end{equation*}
we observe that the key fact about the representation is that the singularity can be actually cancelled out once $u$ lies in any H\"{o}lder space, which can be guaranteed by any control upon
$\|\nabla u\|_{L^{p}}$ with $p>2$.
Accordingly, a term of commutator type which can be written in the above integral form plays an important role in   carrying out further  computations.

Thus, motivated by the    above analysis, we find an alternative approach to estimate the upper bound of the density. Note that equivalently we have for $\theta(\rho)$ as in \eqref{qis1}
\begin{equation*}
\frac{\mathrm{D}}{dt}\theta(\rho)+P=-F,
\end{equation*}
where  $F=(2\mu+\lambda)\mathrm{div}u-P$ is the so-called effective viscous flux.
This implies that in order to estimate $\rho$, we need  to establish proper bound on $F$ which indeed solves a Neumann problem as follows:
\begin{equation}\label{qp18}
\begin{cases}
  \triangle F=\mathrm{div}(\rho \dot{u}) &\mathrm{in}\, \,  \Omega, \\
   \frac{\partial F}{\partial n}=\rho\dot{u}\cdot n &\mathrm{on}\, \,  \partial \Omega.
       \end{cases}
\end{equation}
Generally, we can get a pointwise representation of $F$ via applying Green's function. We first consider the simplest case that $\Omega$ is the unit disc $\mathbb{D}$, where  Green's function takes the form(see \cite{2016Representation}):
\begin{equation}\label{oqw1}
N(x, y)=-\frac{1}{2\pi}\left[\mathrm{log}|x-y|+\mathrm{log}\left||x|y-\frac{x}{|x|} \right|\right].
\end{equation}
After applying Green's identity, we derive the integral representation of $F$, but the singularity is still out of control especially for the integral near the boundary. Fortunately, we observe that for
each $x\in\partial\mathbb{D}$, $x$ is just the outer normal of $\partial\mathbb{D}$, and slip boundary yields $u(x)\cdot x=0$ on $\partial\mathbb{D}$ (see also \eqref{1}). Such cancellation condition helps us reduce the integral
representation to the desired commutator form, thus we may apply the previous idea.

For general domains, unfortunately,   the precise formula of  Green's  function is unknown. So we need   further arguments. The cancelation of singularity on the boundary is indeed a type of geometry restriction, and we do not know whether it remains true in other circumstance or not. However, we observe that the orthogonality which is crucial in our computation is preserved under conformal mapping, since it preserves the angle. Moreover, Riemann mapping theorem (\cite[Chapter 9]{2007Complex}) makes sure every simply connected domain
is conformally   equivalent with the unit disc. Consequently, it is reasonable to expect that we can reduce the general case to that of unit disc via making use of conformal mapping. The goal is achieved
by pull-back Green's function (see \eqref{348}). Finally, after some careful calculations,  we obtain the desired upper bound  upon $\rho$.

Another major technical difficulty arises from lower order estimates,  especially, the trace of  $\nabla u$ on the boundary $\partial \Omega$,  which  seems ought to be controlled by  the $L^p$-norm of $\nabla^2u$ in general case.  Unfortunately,  we have no a priori estimate on $\nabla^2u$ at this stage.  To overcome it,  we mainly adapt the ideas due to Cai-Li\cite{caili01}.  On the one hand, observe that the slip boundary condition $u\cdot n|_{\partial \Omega}=0$ admits
\begin{equation}\label{16}
  u=(u\cdot n^\bot)n^\bot, \,\, (u\cdot\nabla)u\cdot n=-(u\cdot\nabla)n\cdot u,
\end{equation}
where  $n^\bot$ is the unit tangential vector on the boundary $\partial \Omega$ denoted by\be\label{q1w} n^\bot\triangleq (n_2,-n_1).\ee  On the other hand, as observed by  Cai-Li\cite{caili01}, it holds that for   function $f\in H^1$\bnn\ba \left|\int_{\partial \Omega}u\cdot \nabla f ds\right|&=\left| \int_{\partial \Omega}(u\cdot n^\bot)n^\bot\cdot \nabla f ds\right|\\&= \left|\int_{  \Omega} \nabla f\cdot \nabla^\bot (u\cdot n^\bot) dx\right|\le C\|f\|_{H^1}\|u\|_{H^1},\ea\enn where the operator $\nabla^\bot=(\partial_2,-\partial_1) $ is as in \eqref{ou1r}. With these key observations at hand, after some delicate calculations,  all the boundary term including $\nabla u$ caused by integration by parts can be handled. 
  See Lemma \ref{042} for details. Finally, after all obstacles eliminated,  the conclusion follows in a rather routine way.

\begin{remark}
 In order to derive the proper representation on $F$, we spend much effort  on studying the property (see Lemma \ref{003}) of the pull-back Green's function $\tilde{N}$ \eqref{348} and  computing corresponding representation \eqref{qw10}, instead  of    using directly the common   Green's  function $G$ (see \cite{1971The,1972The}),  the precise formula of  which is still unknown for general bounded smooth domains. We comment that such procedure is necessary because the usual Green's  function $G $    can not meet our requirements. Indeed, 
 for usual Green's  function $G$,  the known estimates   we can
use are the order of singularity (see \cite{1971The,1972The}):
\begin{equation*}
|\partial^{\alpha}_x\partial^{\beta}_y G(x,y)|\leq C|x-y|^{-\alpha-\beta},
\end{equation*}
with $\alpha,\beta$   nonnegative integers such that $\alpha+\beta>0$. This information on singularity only guarantees \eqref{qw01} dominated by
\begin{equation}\label{obu1}
\int_\Omega\frac{|u|}{|x-y|^2}\rho |u|(y)dy,
\end{equation}
whose singularity is indeed out of control in dimension 2.
But after using $\tilde{N},$ we change the key integral representation of  the  effective viscous flux $F$ into   \eqref{qw10}  which together with the slip boundary \eqref{1} further  expresses the form of above term \eqref{obu1}    explicitly as
\begin{equation*}
\int_\Omega\frac{|u(x)-u(y)|}{|x-y|^2}\rho |u|(y)dy.
\end{equation*}
Such term     can improve  the integrability just like the common Calderon-type commutator   (see Proposition \ref{006}) and meets our request.
\end{remark}

The rest of paper is organized as follows. In Section \ref{sec1}, we state some elementary preliminaries   which will be used later.  Sections  \ref{sec2} and  \ref{sec3} are devoted to the lower order and higher one a priori estimates which will be carried out in details.  Finally, with all necessary estimates in hand, we will prove the main results, Theorems \ref{thmq1} and \ref{thmq2},  in Section  \ref{sec4} shortly in a standard way.
\section{\label{sec1}Preliminaries}
First of all,  we quote the well-known local existence theory in\cite{weko_39341_1,Solonnikov1980}, where the initial density is strictly away from vacuum. Actually it is the cornerstone of our proof on global situation.
\begin{lemma} \label{u21}
Assume that $(\rho_0,  m_0)$ satisfies
\begin{equation}
\rho_0\in H^2, \inf_{x\in\Omega}\rho_0(x)>0, \ u_0\in \{ H^2| u_0\cdot n=0, {\rm curl}u_0=0 \mbox{  on }\partial \Omega\},\  \ m_0=\rho_0u_0.  \label{21}
\end{equation}
Then there is a small time $T>0$ and a constant $C_0>0$ both depending only on $ \Omega ,\mu,\beta,\gamma,\|\rho_0\|_{H^2}$,  $\|u_0\|_{H^2},$ and $\inf\limits_{x\in\Omega}\rho_0(x)$ such that there exists a unique strong solution $(\rho,  u)$ to the problem \eqref{11}-\eqref{15} in $\Omega\times(0, T)$
satisfying
\begin{equation}
\left\{ \begin{array}{l}
\rho\in C([0, T]; H^2), \ \rho_t\in C([0, T];H^1), \\
u\in L^2(0, T;H^3), \ u_t\in L^2(0, T;H^1), \\   \label{22}
u_t\in L^2(0, T;H^2), \ u_{tt}\in L^2((0, T)\times\Omega),
       \end{array} \right. \\
\end{equation}
and
\begin{equation}
\inf_{(x, t)\in\Omega\times(0, T)}\rho(x, t)\geq C_0 > 0. \label{23}
\end{equation}
\end{lemma}

Next, to deduce the proper estimates upon $\nabla u$,  we require div-curl control which can be found in \cite{Aramaki2014Lp,Mitrea2005Integral,Wahl1992Estimating}.
\begin{lemma}\label{lm22}
Let $1<q<\infty$ and  $\Omega$ be a bounded domain in $\mathbb{R}^2$ with Lipschitz boundary $\partial\Omega$.  For $v\in W^{1, q}$,  if $\Omega$ is simply connected and $v\cdot n=0$ on $\partial\Omega$,  then it holds that
\begin{equation}
\|\nabla v\|_{L^p}\leq C(\|\mathrm{div}v\|_{L^p}+\|\mathrm{curl}v\|_{L^p}). \label{24}
\end{equation}
\end{lemma}

Moreover, for technical reasons, the modified version of classical Poincar\'{e}-Sobolev inequality will be used frequently.
\begin{lemma}[\cite{tal1}]
 There exists a positive constant $C$ depending only on $\Omega$ such that every function $u\in H^1(\Omega)$ satisfies for $2<p<\infty$,
\begin{equation}
\|u\|_{L^p}\leq Cp^{1/2}\|u\|^{2/p}_{L^2}\|u\|^{1-2/p}_{H^1}.  \label{25}
\end{equation}
In particular,   $\|u\|_{H^1}$ can be replaced by $\|\nabla u\|_{L^2}$ provided
$$ u\cdot n|_{\partial\Omega}=0 \mbox{ or }\int_\Omega udx=0.  $$


\end{lemma}

Next,  to obtain the estimate on the $L^\infty(0,T;L^p(\Omega))$-norm of $\nabla \rho,$  we need the following Beale-Kato-Majda type inequality  which was first proved in \cite{beale1984remarks} when $\mathrm{div}u=0.$
\begin{lemma}[\cite{beale1984remarks,caili01}]\label{qo01}
For $2<q<\infty$,  assume that  $  u\in W^{2, q}(\Omega)$ with $u\cdot n=0$ and $\mathrm{curl}u=0$ on $\partial\Omega$. Then  there is a constant $C=C(q)$ such that the following estimate holds
\begin{equation}
\|\nabla u\|_{L^\infty}\leq C(\|\mathrm{div}u\|_{L^\infty}+\|\mathrm{curl}u\|_{L^\infty})\mathrm{log}(e+\|\nabla^2u\|_{L^q})+C\|\nabla u\|_{L^2}+C.
\end{equation}
\end{lemma}

 Next, let $\Omega$ be as in Theorem \ref{thmq1}  and $\mathbb{D}$ the unit disc. Then, by Riemann mapping theorem (\cite[Chapter 9]{2007Complex}), there exists a  conformal mapping   $\varphi=(\varphi_1, \varphi_2):\bar{\Omega}\rightarrow\bar{\mathbb{D}}$
  with  $\varphi_1$ and $\varphi_2$ satisfying the following Cauchy-Riemann equations:
\begin{equation}\label{001}
\begin{cases}
\partial_1\varphi_1=\partial_2\varphi_2, \\
\partial_2\varphi_1=-\partial_1\varphi_2.
       \end{cases}
\end{equation}
Moreover, the conformal mapping $\varphi(x)$ shares the following crucial properties: 
\begin{lemma}[\cite{2007Complex,1961On,Warschawski1935On}]\label{333}
The conformal mapping $\varphi(x):\bar{\Omega}\rightarrow\bar{\mathbb{D}}$ is smooth   and satisfies

(i) $\varphi(x)$ is a one to one holomorphic mapping from $\Omega$ to $\mathbb{D}$ and maps the boundary $\partial\Omega$ onto the boundary $\partial\mathbb{D}$.

(ii) For any integer $k=0,1,2$, there exists some constant $C$ depending only on $\Omega$ and $k$ such that
\begin{equation}\label{3353}
|\nabla^k\varphi(x)|\leq C,\,\,\, \forall x\in\bar{\Omega}.
\end{equation}
Consequently,
\begin{equation}\label{3354}
|\nabla^k\varphi(x)-\nabla^k\varphi(y)|\leq C|x-y|,\,\,\, \forall x, y\in\bar{\Omega}.
\end{equation}

(iii) There exist  two positive constants $c_1, c_2$ such that
\begin{equation}\label{3355}
c_1|x-y|\leq|\varphi(x)-\varphi(y)|\leq c_2|x-y|.
\end{equation}

(iv) Angle is preserved by $\varphi(x)$, that is, for any two smooth curves $\gamma_1(t), \gamma_2(t):(0, 1)\rightarrow\bar{\Omega}$,
\begin{equation*}
\langle\frac{d}{dt}\gamma_1(t), \frac{d}{dt}\gamma_2(t)\rangle
=\langle\frac{d}{dt}\varphi(\gamma_1(t)), \frac{d}{dt}\varphi(\gamma_2(t))\rangle,\,\,\, \forall t\in(0, 1).
\end{equation*}

(v) For any harmonic function $h(y)$ in $\mathbb{D}$, $h(\varphi(x))$ is still harmonic in $\Omega$, that is,
\begin{equation*}
\Delta h(y)=0\ \mathrm{in}\ \mathbb{D}\Rightarrow\Delta h(\varphi(x))=0\ \mathrm{in}\ \Omega.
\end{equation*}
\end{lemma}

Finally, since conformal mapping preserves angles, one immediately has the following conclusion.
\begin{lemma}\label{334}Let $n$   be the unit outer normal vector of $\partial\Omega$ at $x_0.$  Then
\be  \label{qi01}
n\cdot\nabla\varphi(x_0)=|\nabla\varphi_1(x_0)|\,\widetilde{n},
\ee
  where $\widetilde{n}$ is the unit outer normal vector  of $\partial\mathbb{D}$ at $\varphi(x_0).$
\end{lemma}
\begin{proof}  Let $\gamma(s):[-1, 0]\rightarrow\bar\Omega$ be a smooth curve satisfying
\begin{equation*}
\gamma(s)\in\Omega\,\, \text{for any} \,\,s\in [-1, 0),\,\, \text{and}\,\,\gamma(0)=x_0, \,\,\, \frac{d}{ds}\gamma(s)|_{s=0}=n.
\end{equation*}
 Suppose $\gamma_0(s):(-1, 1)\rightarrow\partial\Omega$ is a curve which lies entirely on the boundary $\partial\Omega$ satisfying
\begin{equation*}
\gamma_0(0)=x_0,\,\, \frac{d}{ds}\gamma_0(s)|_{s=0}=v.
\end{equation*}
Then $v$ is the tangent vector of $\partial\Omega$ at $x_0$, and
\begin{equation*}
\langle \frac{d}{ds}\gamma(s), \frac{d}{ds}\gamma_0(s)\rangle|_{s=0}=\langle n, v\rangle=0.
\end{equation*}
By Lemma \ref{333}, we have
\begin{equation*}
\langle\frac{d}{ds}\varphi(\gamma(s)), \frac{d}{ds}\varphi(\gamma_0(s))\rangle|_{s=0}=\langle\frac{d}{ds}\gamma(s), \frac{d}{ds}\gamma_0(s)\rangle|_{s=0}
=0,
\end{equation*}
which implies that $\frac{d}{ds}\varphi(\gamma(s))|_{s=0}$ is the normal vector of $\partial\mathbb{D}$ at $\varphi(x_0)$ since
$\frac{d}{ds}\varphi(\gamma_0(s))|_{s=0}$ is just the tangent vector of $\partial\mathbb{D}$ at $\varphi(x_0)$.

we check that
\begin{equation*}\ba
\frac{d}{ds}\varphi(\gamma(s))|_{s=0} =(\frac{d}{ds}\gamma(s)|_{s=0}\cdot\nabla)\varphi(x_0) =n\cdot\nabla\varphi(x_0),\ea
\end{equation*}
and by Cauchy-Riemann equations \eqref{001},
\begin{equation*}
\begin{split}
|n\cdot\nabla\varphi|^2&=\big((n\cdot\nabla)\varphi_1\big)^2+\big((n\cdot\nabla)\varphi_2\big)^2\\
&=(n_1\cdot\partial_1\varphi_1)^2+2n_1n_2\cdot\partial_1\varphi_1\partial_2\varphi_1+(n_2\cdot\partial_2\varphi_1)^2+\\
&\quad(n_1\cdot\partial_1\varphi_2)^2+2n_1n_2\cdot\partial_1\varphi_2\partial_2\varphi_2+(n_2\cdot\partial_2\varphi_2)^2\\
&=|\partial_1\varphi_1|^2+|\partial_2\varphi_1|^2=|\nabla\varphi_1|^{2}=|\nabla\varphi_2|^{2}.
\end{split}
\end{equation*}
As a result,
\begin{equation*}
n\cdot\nabla\varphi(x_0)=|\nabla\varphi_1(x_0)|\,\widetilde{n},
\end{equation*}
where $\widetilde{n}$ is the unit outer normal of $\mathbb{D}$ at $\varphi(x_0)$.
\end{proof}

\begin{remark}
With the boundary condition $u\cdot n=0$ on $\partial\Omega$, we can check that for any $y_0\in\partial\Omega$, $\varphi(y_0)$ and $(u\cdot\nabla)\varphi(y_0)$ are the outer normal and tangent vector at $\varphi(y_0)$ on $\partial\mathbb{D}$ respectively. Consequently,\begin{equation}
u_i(y_0)\cdot\partial_i\varphi_j(y_0)\cdot\varphi_j(y_0)=\langle(u\cdot\nabla)\varphi(y_0), \varphi(y_0)\rangle=0,\label{1}
\end{equation}
which is an important observation for further work.
\end{remark}

\section{\label{sec2}A priori estimates (I): upper bound of the density}

 In this section and the next,  we  always assume that $(\rho, u)$ is the strong solution to \eqref{11}-\eqref{15} on $\Omega\times(0, T)$ whose existence is guaranteed by Lemma \ref{u21}.

We define the effective viscous flux $F$  and the vorticity $\omega$ in usual manner:
\begin{equation}\label{303}
 F\triangleq(2\mu+\lambda(\rho))\mathrm{div}u- P , \, \,
\omega\triangleq\nabla^\bot\cdot u=\partial_2u_1-\partial_1u_2.
\end{equation}

Then, we set
\begin{equation}
A_1^2(t)\triangleq 1+\int_\Omega\bigg(\omega^2(t)+\frac{F^2(t)}{2\mu+\lambda(\rho(t))}\bigg)dx,
\end{equation}
\begin{equation}
A_2^2(t)\triangleq\int_\Omega\rho(t)|\dot{u}(t)|^2dx,  \label{305}
\end{equation}
and
\begin{equation}
R_T\triangleq 1+\sup_{0\leq t\leq T}\|\rho(t)\|_{L^\infty}.
\end{equation}

We now state the standard energy estimate.
\begin{lemma}\label{3111}
There exists a positive constant $C$ depending only on $  \gamma, \|\rho_0\|_{L^\gamma},$   and $\|\sqrt{\rho_0}u_0\|_{L^2}$ such that
\begin{equation}
\sup_{0\leq t\leq T}\int_\Omega(\rho|u|^2+\rho^\gamma)dx+\int_0^T\int_\Omega\left( (2\mu+\lambda(\rho ))(\mathrm{div}u  )^2+\mu\omega^2 \right)dxdt\leq C.  \label{30}
\end{equation}
\end{lemma}
\begin{proof}
It is easy to check that
\begin{equation}
P_t+\mathrm{div}(Pu)+(\gamma-1)P\mathrm{div}u=0,  \label{31}
\end{equation}
due to $\eqref{11}_1$.  Then,  integrating \eqref{31} over $\Omega$ and using slip boundary condition \eqref{15} one gets
\begin{equation*}
\frac{d}{dt}\int_\Omega \frac{\rho^\gamma}{\gamma-1}dx+\int_\Omega P\mathrm{div}udx = 0.
\end{equation*}
Since $\Delta u=\nabla \mathrm{div}u+\nabla^\bot\omega$,  we rewrite the equation of conservation of momentum $\eqref{11}_2$ as
\begin{equation}
\rho\dot{u}+\nabla P=\nabla((2\mu+\lambda(\rho)) \mathrm{div}u)+\mu\nabla^\bot\omega.  \label{32}
\end{equation}
Multiplying \eqref{32} by $u$ and integrating over $\Omega$,  together with the boundary condition \eqref{15},  one gets
\bnn
\frac{d}{dt}\int_\Omega\left(\frac12\rho|u|^2+ \frac{\rho^\gamma}{\gamma-1}\right)dx+\int_\Omega\left( (2\mu+\lambda(\rho ))(\mathrm{div}u  )^2+\mu\omega^2 \right)dx=0.
\enn
Integrating this over $(0, T)$ yields \eqref{30}.
\end{proof}

Next, we state a known result
concerning the  estimate on the $L^\infty(0,T;L^p(\Omega))$-norm of the density  whose proof is similar to that of  \cite[(36)]{vaigant1995}.
 \begin{lemma}[\cite{vaigant1995}]\label{004}
Let $\beta >1.$ Then, for any $1\leq p<\infty$, there is a constant $C$ depending on $ \Omega, T,\mu,\beta,\gamma,\|\rho_0\|_{L^{\infty}}$, and $\|u_0\|_{H^1}$,
such that
\begin{equation}\label{qp80}
\|\rho\|_{L^{p}}\leq C(T)p^{\frac{2}{\beta-1}},
\end{equation}
where and in what follows, we  use the convention that $C$ denotes a generic positive constant depending on  $ \Omega, T,\mu,\beta,\gamma,\|\rho_0\|_{L^{\infty}}$, and $\|u_0\|_{H^1}$, and we write $C(p)$ to emphasize that $C$ depends on $p.$
\end{lemma}

We also need extra integrability up on the momentum $\rho u$ where we modify the  proof   of \cite[Lemma 3.7]{huang2016existence} slightly due to  the boundary effect.
\begin{lemma}  \label{0001} There exists some  suitably small generic constant  $\nu_0\in (0,1)$   which depends only on $\mu$ and $\Omega$ such that
for \be\label{nu1q}\nu\triangleq R_T^{-\frac{\beta}{2}}\nu_{0},\ee   there is a constant $C $ depending on $\Omega$,  $T,\mu,\beta,\gamma, \|\rho_0\|_{L^{\infty}}$, and $\|u_0\|_{H^1}$ such that
\begin{equation}\label{qp90}
\sup_{0\leq t\leq T}\int_\Omega\rho |u|^{2+\nu}dx\leq C .
\end{equation}
\end{lemma}
\begin{proof}  First, combining   \eqref{24}, \eqref{qp80}, and Poincar\'{e}'s inequality gives
\begin{equation}
 \|u\|^2_{H^1}\leq C\|\nabla u\|^2_{L^2}\leq CA_1^2(t) \leq CR_T^\beta \|\nabla u\|^2_{L^2}+C,  \label{310}
\end{equation}
which together with \eqref{30} and \eqref{qp80} implies \be \label{qp87}\int_0^T\left(\|  u\|_{H^1}^2+A_1^2\right)dt \le C.\ee

Then, following the  proof   of \cite[Lemma 3.7]{huang2016existence}, multiplying \eqref{32} by $ |u|^\nu u$   and integrating the resulting equality over $\Omega$ by parts, we arrive at
\begin{equation*}
\begin{split}
& \frac{1}{2+\nu}\frac{d}{dt}\int\rho |u|^{2+\nu}dx+ \int|u|^\nu\big((2\mu+\lambda)(\mathrm{div}u)^2+\mu\omega^2\big)dx\\
&\leq  \nu\int(2\mu+\lambda)|\mathrm{div}u||u|^{\nu}|\nabla u|dx+ \nu\mu\int|u|^{\nu}|\nabla u|^2dx+C
\int\rho^\gamma|u|^\nu|\nabla u|dx\\
&\leq \frac{1}{2}\int |u|^\nu(\mathrm{div}u)^2dx+\frac{\nu_0^2(\mu+1)}{2}\int|u|^\nu|\nabla u|^2dx+\int|\nabla u|^{2}dx\\
&\quad+C\int\rho|u|^{2+\nu}dx +C\int\rho^{\frac{4+2\nu}{2-\nu}\gamma-\frac{2\nu}{2-\nu}}dx,
\end{split}
\end{equation*}
which together with   \eqref{qp87}, Lemma \ref{004}, and the following  Lemma \ref{a004} yields \eqref{qp90} and   finishes the  proof of Lemma \ref{0001}.
\end{proof}

\begin{lemma}\label{a004}
 There exist positive constants $\tilde \nu$ and $  C$  both depending only on $\Omega$ such that for any $\nu\in (0,\tilde \nu)$,
\begin{equation}
\int|u|^{\nu}|\nabla u|^2dx\leq C\int|u|^{\nu}\big((\mathrm{div}u)^2+\omega^2)dx.
\end{equation}
\end{lemma}
\begin{proof} First, direct computation  shows
\begin{equation}\label{qa3}
  \frac{1}{2}|u|^\nu |\nabla u|^2 \le |\nabla(|u|^{\frac{\nu}{2}} u)|^2+ \nu^2|u|^\nu|\nabla u|^2.
\ee

Next, let us calculate that
\begin{equation}\label{qa1}\ba
(\mathrm{div}(|u|^{\frac{\nu}{2}} u))^2&=\big(|u|^{\frac{\nu}{2}} \mathrm{div} u-\nabla |u|^{\frac{\nu}{2}}\cdot u\big)^2\\&
\leq C|u|^\nu \big(\mathrm{div} u\big)^2+ C\nu^2|u|^\nu|\nabla u|^2 .\ea
\end{equation}
Similarly, we have
\begin{equation}\label{qa2}
(\mathrm{curl}(|u|^{\frac{\nu}{2}} u))^2
\leq C|u|^\nu \big(\mathrm{curl} u\big)^2+C\nu^2|u|^\nu|\nabla u|^2.
\end{equation}

Finally, since $\left.|u|^{\frac{\nu}{2}}u\cdot n\right|_{\partial\Omega}=0$, we apply \eqref{24} to deduce
\begin{equation*}
\int_\Omega\left|\nabla (|u|^{\frac{\nu}{2}}u)\right|^2dx\leq
C(\Omega)\int_\Omega\left(\left|\mathrm{div}(|u|^{\frac{\nu}{2}}u)\right|^2 +\left|\mathrm{curl}(|u|^{\frac{\nu}{2}}u)\right|^2\right)dx,
\end{equation*} which together with  \eqref{qa3}--\eqref{qa2}  proves   Lemma \ref{a004}.
\end{proof}

Next, the following typical estimates on $\nabla u$ will be used frequently.
\begin{lemma}\label{005} For $p>2 $ and $\varepsilon>0,$ there exists some positive constant $C(p,\varepsilon)$ such that
\begin{equation}\label{qp31}
\|\nabla u\|_{L^{p} }\leq C(p, \varepsilon)R_T^{\frac{1}{2}-\frac{1}{p}+\varepsilon} A_1  \left(\frac{A_2^2}{ A_1^2}\right)^{\frac{1}{2}-\frac{1}{p}}
+C(p, \varepsilon)R_T^{\varepsilon} A_1 .
\end{equation}
\end{lemma}
\begin{proof}
First, we rewrite the momentum equations as
\begin{equation}
\rho\dot{u}=\nabla F+\mu\nabla^\bot\omega,   \label{34}
\end{equation} which together with  the boundary condition \eqref{15} yields that $F$ solves the Neumann problem \eqref{qp18} and that   $\omega$ solves the related Dirichlet problem:
\begin{equation}
\begin{cases}
  \mu\Delta\omega=\nabla^\bot\cdot(\rho\dot{u})& \mbox{ in } \Omega, \\
   \omega=0& \mbox{ on } \partial\Omega.  \end{cases}
\end{equation}

Then, standard $L^p$ estimate of elliptic equations (see\cite{gilbarg2015elliptic}, \cite[Lemma 4.27]{novotny2004introduction}) implies that for $k\geq 0$ and  $p\in(1, \infty)$,
\begin{equation}
\|\nabla F\|_{W^{k, p}}+\|\nabla w\|_{W^{k, p}}\leq C(p, k)\|\rho\dot{u}\|_{W^{k, p}}.  \label{336}
\end{equation}
In particular,  we have
\bnn
\|\nabla F\|_{L^2}+\|\nabla w\|_{L^2}\leq C\|\rho\dot{u}\|_{L^2}\leq CR_T^{1/2}A_2,
\enn
which together with the Poincar\'{e} inequality and \eqref{qp80} yields
\begin{equation} \label{311} \ba
\|F\|_{H^1}+\|\omega\|_{H^1}&\leq C(\|\nabla F\|_{L^2}+\|\nabla w\|_{L^2})+\frac{C}{|\Omega|}\int{F}dx\\& \leq C R_T^{1/2}A_2+CA_1 . \ea
\end{equation}

Finally, according to \eqref{24} and Lemma \ref{004}, we get
\begin{equation}\label{qp91}
\begin{split}
\|\nabla u\|_{L^{p} }&\leq C(p)\big(\|\mathrm{div}u\|_{L^{p} }+\|\omega\|_{L^{p} }\big) \\
&\leq C(p) \left\|\frac{F}{2\mu+\lambda}\right\|_{L^{p} }+C(p) \|\omega\|_{L^{p} }+C(p) \\ &\le C(p)    \left\|\frac{F}{2\mu+\lambda}\right\|_{L^{2} }^{\frac{2}{p}-\varepsilon}\|F\|_{L^{\frac{ 2(1+\varepsilon)p-4}{p\varepsilon}} }^{-\frac{2}{p}+1+\varepsilon} +C(p)\|\omega\|_{L^{p}} +C(p)
\\&\le  C(p, \varepsilon)A_1^{\frac{2}{p}-\varepsilon}\|F\|_{L^{2} }^{\varepsilon} \|F\|_{H^1 }^{1-\frac{2}{p}}+C(p)A_1^{\frac{2}{p}}\| \omega\|_{H^1}^{1-\frac{2}{p}}+C(p)\\& \leq C(p, \varepsilon)R_T^{\frac{\beta\varepsilon}{2} }A_1^{\frac{2}{p} }  ( \|F\|_{H^1}+\|\omega\|_{H^1})^{1-\frac{2}{p}} +C(p),
\end{split}
\end{equation}  which together with \eqref{311} gives \eqref{qp31} and finishes the proof of Lemma \ref{005}.
\end{proof}

At present stage,  we are in a position to prove the following crucial estimate on the upper bound of $\mathrm{log}(1+\|\nabla u\|_{L^2})$ in terms of $R_T$ which turns out to be crucial   in obtaining the upper bound of the density.
\begin{proposition}\label{33}
For any $\alpha\in(0, 1)$,  there is a constant $C(\alpha)$ depending only on $\alpha,   \Omega , T, \mu,  \beta,  \gamma, $ $ \|\rho_0\|_{L^\infty}$,  and $\|u_0\|_{H^1}$ such that
\begin{equation}\label{qp40}
\sup_{0\leq t\leq T}\mathrm{log} A^2_1(t) +\int^T_0 \frac{A^2_2(t)}{ A^2_1(t)}dt\leq C(\alpha)R_T^{1 +\alpha }.
\end{equation}
\end{proposition}

\begin{proof} We modify the  ideas of \cite{huang2016existence} to overcome the difficulties arising from the boundary. First, direct calculations show that
\begin{equation}
\nabla^\bot\cdot\dot{u} = \frac{D}{Dt}\omega-(\partial_1u\cdot\nabla)u_2+(\partial_2u\cdot\nabla)u_1=\frac{D}{Dt}\omega+\omega \mathrm{div}u,  \label{35}
\end{equation}
and that
\begin{equation}
\begin{split}
\mathrm{div}\dot{u}=&\frac{D}{Dt}\mathrm{div}u+(\partial_1u\cdot\nabla)u_1+(\partial_2u\cdot\nabla)u_2\\ \label{36}
=&\frac{D}{Dt}\bigg(\frac{F}{2\mu+\lambda}\bigg)+\frac{D}{Dt}\bigg(\frac{P }{2\mu+\lambda}\bigg)-2\nabla u_1\cdot\nabla^\bot u_2 + (\mathrm{div}u)^2.
\end{split}
\end{equation}

Then,
multiplying \eqref{34} by $2\dot{u}$ and integrating the resulting equality over $\Omega$,  we obtain after using
\eqref{35},  \eqref{36}, and the boundary condition \eqref{15} that
\begin{equation}\label{qiu23}
\begin{split}
\frac{d}{dt}A^2_1+2A^2_2=&-\mu\int_\Omega\omega^2\mathrm{div}udx+4\int_\Omega F\nabla u_1\cdot \nabla^\bot u_2dx-2\int_\Omega F(\mathrm{div}u)^2dx\\
&-\int_\Omega\frac{(\beta-1)\lambda-2\mu}{(2\mu+\lambda)^2}F^2\mathrm{div}udx+ 2\beta\int_\Omega\frac{\lambda P  }{(2\mu+\lambda)^2}F\mathrm{div}udx\\
&-2\gamma\int_\Omega\frac{P}{ 2\mu+\lambda }F\mathrm{div}udx
 +\int_{\partial\Omega}Fu\cdot\nabla u\cdot nds\triangleq\sum_{i=1}^7I_i.
\end{split}
\end{equation}

Now we estimate each $I_i$ as follows:

First, combining \eqref{25},   \eqref{310}, and  the H\"{o}lder inequality leads to
\begin{equation}\label{qiu24}\ba
I_1&\leq C\|\omega\|_{L^4}^2\|\mathrm{div}u\|_{L^2}
\\& \leq C \|\omega\|_{L^2}\|\nabla\omega\|_{L^2}A_1\\&\le  CA_1^2\|\nabla\omega\|_{L^2} .\ea
\end{equation}

Next, $I_2$ is the most difficult term which requires careful calculations. The major obstacle in the bounded domain is that we can't make use of common pairing between BMO and Hardy space   used by  \cite{huang2016existence}. We take alternative approach as follows:

For $\alpha\in (0,1),$ letting  $p\ge 8\beta/\alpha ,$ we use \eqref{qp91}, H\"older's and Sobolev's inequalities to get
\begin{equation}
\begin{split}|I_2|&\le
\int   F|\nabla u|^2 dx\\&\leq\|F\|_{L^{p}}\|\nabla u\|_{L^{2}}^{\frac{2(p-3)}{p-2}}\|\nabla u\|_{L^{p}}^{\frac{2}{p-2}} \\&\leq C(p,\varepsilon)\|F\|_{L^{p}}A_1^{\frac{2(p-3)}{p-2}}\left(R_T^{\frac{\beta\varepsilon}{2} }A_1^{\frac{2}{p} }  ( \|F\|_{H^1}+\|\omega\|_{H^1})^{1-\frac{2}{p}}+1\right)^{\frac{2}{p-2}} 
\\
&\leq C(\alpha) R_T^{\frac{\alpha\beta}{ 2}}\left(\|F||_{H^1} +\|\omega||_{H^1} \right) A_1^2  +C(\alpha)  A_1^2.
\end{split}
\end{equation}
where in the last line, we have used  $ \|F\|_{L^{p}}\le  C(p) \|F\|_{L^{2}}^{\frac{2}{p }}\|F\|_{H^1}^{1-\frac{2}{p }}$  and $\|F\|_{L^2}\leq CR_T^{\beta/2}A_1.$

Next,  the H\"{o}lder inequality yields that for $0 <\alpha< 1$,
\begin{equation}
\begin{split}
\sum_{i=3}^6I_i&\leq C\int_\Omega\frac{F^2|\mathrm{div}u|}{2\mu+\lambda}dx+C\int_\Omega\frac{P}{2\mu+\lambda}|F||\mathrm{div}u|dx\\
&\leq C A_1 \left\|\frac{F^2}{2\mu+\lambda}\right\|_{L^2}+C A_1  \|F\|_{L^{2+4\gamma/\beta}} \\&\leq C A_1 \bigg\|\frac{F}{(2\mu+\lambda)^{1/2}}\bigg\|_{L^2}^{1-\alpha} \|F\|^{1+\alpha}_{L^{2(1+\alpha)/\alpha}}+C(\alpha) A_1  \|F\|_{H^1}\\&\leq C (\alpha) A_1^{2-\alpha}  \|F \|_{L^2}^{ \alpha} \|F\|_{H^1} +C(\alpha) A_1  \|F\|_{H^1} \\
&\leq C(\alpha)R_T^{\frac{\alpha\beta}{2}} A_1^2 \|F\|_{H^1} ,
\end{split}
\end{equation}
where in the last line, we have used   $\|F\|_{L^2}\leq CR_T^{\beta/2}A_1$.

Finally,  thanks to \eqref{16},  we deal with $I_7$ via:
\begin{equation}
|I_7|=\left|\int_{\partial\Omega} F(u\cdot \nabla n\cdot u) ds\right| \leq C\|F\|_{H^1}\|\nabla u\|_{L^2}^2\leq C\|F\|_{H^1}A_1^2.  \label{332}
\end{equation}
Putting all the estimates \eqref{qiu24}--\eqref{332} into \eqref{qiu23}
 yields that for any $\alpha\in(0, 1)$,
\be\label{312}
\ba
\frac{d}{dt}A^2_1(t)+A^2_2(t)&\leq  C(\alpha) R_T^{\frac{\alpha\beta}{ 2}}\left(\|F||_{H^1} +\|\omega||_{H^1} \right) A_1^2  +C(\alpha)  A_1^2 \\ &\leq  C(\alpha) R_T^{\frac{1+\alpha\beta}{ 2}}A_2 A_1^2  +C(\alpha)  A_1^3 \\
&\leq C(\alpha)R_T^{1+ \alpha\beta}   A_1^4  +\frac{1}{2}A_2^2,
\ea
\ee where in the second inequality we have used \eqref{311}. Combining this with \eqref{qp87}  proves \eqref{qp40} and finishes the proof of Proposition \ref{33}.
\end{proof}

With all preparation done, we turn to the crux of our problem which is to estimate the effective viscous flux $F  $ defined as in \eqref{303}.
Such term can be estimated via solving the Neumann boundary problem \eqref{qp18}
with the help of classical Green's function. 
However, except for some special domains, it is difficult to write down Green's function for arbitrary domains $\Omega$ in an explicit form, which is a crucial obstacle in further calculations.

To get over it, we start with the simplest case, the unit disc $\mathbb{D}$.
We note that Green's  function of the unit disc (see \cite{2016Representation}) takes the form \eqref{oqw1}.
It meets our requirements perfectly and provide a prototype of our whole proof.
Moreover, thanks to Riemann mapping theorem \cite[Chapter 9]{2007Complex}, every simply connected domain is conformally equivalent with the unit disc.
So we may pull back Green's function of the unit disc via conformal mapping, and reduce the general case to that of unit disc in some sense.

Precisely, let  $\varphi=(\varphi_1, \varphi_2):\bar{\Omega}\rightarrow\bar{\mathbb{D}}$ be the conformal mapping   which  satisfies \eqref{001} and Lemma \ref{333}.
We then define the pull-back Green's function $\widetilde{N}$ of $\Omega$ as:
\begin{equation}\ba\label{348}
\widetilde{N}(x, y)&\triangleq N(\varphi(x), \varphi(y))
\\&=-\frac{1}{2\pi}\left[\mathrm{log}|\varphi(x)-\varphi(y)|+ \mathrm{log}\left||\varphi(x)|\varphi(y)-\frac{\varphi(x)}{|\varphi(x)|}\right|
\right].\ea
\end{equation}
Since the outer normal derivative of $\widetilde{N}$ on the boundary $\partial\Omega$ is no longer constant which is illustrated by the next Lemma \ref{003},
 $\widetilde{N}$ is   not the ``real" Green's function of $\Omega$ in the classical sense  but still sufficient for our further calculations.

\begin{lemma}\label{003} Let $n$ be the unit outer normal at $y_0\in\partial\Omega$ and $\widetilde{N}=\widetilde{N}(x, y)$ be given by \eqref{348}, then
$\frac{\partial\widetilde{N}}{\partial n}(x, y_0)=-\frac{1}{2\pi}|\nabla\varphi_1(y_0)|$.
\end{lemma}
\begin{proof}
By Lemma \ref{333}, it is clear that $\varphi(y_0)\in\partial\mathbb{D}$.
A direct calculation shows that
\begin{equation*}
\begin{split}
\frac{\partial}{\partial n}\widetilde{N}(x, y_0)&=n\cdot\nabla_y\widetilde{N}(x, y)|_{y=y_0}\\&=n\cdot\nabla_y N(\varphi(x), \varphi(y)) |_{y=y_0}\\
&=\sum_{i, j=1}^{2}n_i\cdot\partial_i\varphi_j \partial_j N (\varphi(x), \varphi(y))|_{y=y_0}\\
&=(n\cdot\nabla)\varphi_j\cdot \partial_j N (\varphi(x), \varphi(y))|_{y=y_0}\\
&=|\nabla\varphi_1|  \widetilde{n}\cdot\nabla N (\varphi(x), \varphi(y_0))\\
&=-\frac{1}{2\pi}|\nabla\varphi_1|,
\end{split}
\end{equation*} where in the fifth equality we have used \eqref{qi01}.
\end{proof}
  The pull-back Green's  function $\widetilde{N}$ defined as in \eqref{348} can be used to give a pointwise representation of $F$ in $\Omega$ via Green's identity as follows:
\begin{lemma}\label{u34}
Let $F\in C^1(\bar\Omega)\cap C^2(\Omega)$ solve the problem \eqref{qp18}.  Then there holds for $x\in\Omega$
\begin{equation} \label{qp11}
\ba
F(x)
=  -\int_\Omega \nabla_y\widetilde{N}(x, y)\cdot\rho\dot{u}(y)dy+\int_{\partial\Omega}\frac{\partial \widetilde{N}}{\partial n}(x, y)F(y)dS_y.
\ea
\end{equation}
\end{lemma}
\begin{proof}
Denote $\widetilde{N}(x, y)=N_1(x, y)+N_2(x, y)$,  where
\begin{equation*}
N_1(x, y)=-\frac{1}{2\pi}\mathrm{log}|\varphi(x)-\varphi(y)|,  \ N_2(x, y)=-\frac{1}{2\pi}\mathrm{log}\bigg||\varphi(x)|\varphi(y)-\frac{\varphi(x)}{|\varphi(x)|}\bigg|.
\end{equation*}
 By Lemma \ref{333}, conformal mapping preserves harmonicity which implies that $N_2(x, \cdot)$ is harmonic in $\Omega$. Moreover, since  $N_2(x, \cdot)$  has no singular point in $\Omega$,  we apply Green's second identity to $F$ and $N_2(x, \cdot)$ in $\Omega$:
\begin{equation}\label{qp12}
-\int_{}N_2\Delta Fdy=\int_{\partial\Omega}\left(\frac{\partial N_2}{\partial n}F-N_2\frac{\partial F}{\partial n}\right)dS_y.
\end{equation}

Up on applying  Green's second identity to $F$ and $N_1(x, \cdot)$ in $\Omega\backslash \widetilde{B}_r(x)$, where $\widetilde{B}_r(x)=\varphi^{-1}(B_r(\varphi(x)))$ is just the inverse image under $\varphi$ of the ball centered at $\varphi(x)$ of radius $r$, we deduce that for $r$ small enough
\begin{equation}\nonumber
\ba&
\int_{\Omega\backslash \widetilde{B}_r(x)}(\Delta N_1F-N_1\Delta F)dy\\&=\int_{\partial\Omega}\left(\frac{\partial N_1}{\partial n}F-N_1\frac{\partial F}{\partial n}\right)dS_y
-\int_{\partial\widetilde{B}_r(x)}\left(\frac{\partial N_1}{\partial n}F-N_1\frac{\partial F}{\partial n}\right)dS_y.
\ea
\end{equation} Letting $r\rightarrow 0$ and using the fact that
  $\Delta N_1=0$ in $\Omega\backslash \widetilde{B}_r(x), $  we have
\begin{equation} \label{qp14}\ba-
\int_{\Omega} N_1\Delta Fdy=&\int_{\partial\Omega}\left(\frac{\partial N_1}{\partial n}F-N_1\frac{\partial F}{\partial n}\right)dS_y\\&-\lim_{r\rightarrow0}\int_{\partial\widetilde{B}_r(x)}\left(\frac{\partial N_1}{\partial n}F-N_1\frac{\partial F}{\partial n}\right)dS_y.
\ea\end{equation}
For $r$ small enough,
\begin{equation}\label{qp15}
\ba&
\bigg|\int_{\partial\widetilde{B}_r(x)}N_1\frac{\partial F}{\partial n}dS_y\bigg|\\&\leq\int_{\partial\widetilde{B}_r(x)}|N_1(x, y)|dS_y\cdot \sup_{\partial\widetilde{B}_r(x)}|\nabla F|\\
&\leq\int_{\partial B_r(\varphi(x))}\left|\mathrm{log} {|\varphi(x)-\varphi(y)|}\right| \frac{1}{|\nabla\varphi_1(y)|}dS_{\varphi(y)} \cdot\sup_{\partial\widetilde{B}_r(x)}|\nabla F|\\
&\leq Cr|\mathrm{log}r|\cdot\sup_{\partial\widetilde{B}_r(x)} |\nabla\varphi_1|^{-1}\sup_{\partial\widetilde{B}_r(x)}|\nabla F| \rightarrow 0\ as\ r\rightarrow 0.
\ea
\end{equation}
and
\begin{equation}\label{qp13}
\ba&
\int_{\partial\widetilde{B}_r(x)}F\frac{\partial N_1}{\partial n}dS_y\\&=\int_{\partial\widetilde{B}_r(x)}\vec{n}_y\cdot\nabla_yN_1(x, y) F(y)dS_y\\&
= -\frac{1}{2\pi  }\int_{\partial\widetilde{B}_r(x)} \frac{\nabla\varphi(y)[\varphi(y)-\varphi(x)]}{|\nabla\varphi_1(y)| |\varphi(x)-\varphi(y)|} \cdot \frac{\nabla\varphi(y)[\varphi(x)-\varphi(y)]}{|\varphi(x)-\varphi(y)|^2} F(y)dS_y\\
&=-\frac{1}{2\pi r}\int_{\partial\widetilde{B}_r(x)}|\nabla\varphi_1(y)|F(y)dS_y\\
&=-\frac{1}{2\pi r}\int_{\partial B_r(\varphi(x))}F(\varphi^{-1}(\tilde{y}))dS_{\tilde{y}}\rightarrow F(x)\ as \ r\rightarrow 0,
\ea
\end{equation}
where we use \eqref{001} and the same method as in previous Lemma \ref{003}.  Adding \eqref{qp12} and \eqref{qp14} together,  we have   by using \eqref{qp15} and \eqref{qp13}  \begin{equation}\nonumber
\begin{split}
F(x)=&\int_\Omega \widetilde{N}(x, y)\mathrm{div}(\rho\dot{u})(y)dy-\int_{\partial\Omega}\widetilde{N}(x, y)\rho\dot{u}\cdot\vec{n}(y)dS_y\\
&+\int_{\partial\Omega}\frac{\partial \widetilde{N}}{\partial n}(x, y)F(y)dS_y\\\label{01}
=&-\int_\Omega \nabla_y\widetilde{N}(x, y)\cdot\rho\dot{u}(y)dy +\int_{\partial\Omega}\frac{\partial \widetilde{N}}{\partial n}(x, y)F(y)dS_y,
\end{split}
\end{equation} which gives \eqref{qp11} and finishes the proof of Lemma \ref{u34}.
\end{proof}

Now the central point is to estimate $F$ defined by \eqref{qp11}.  For the boundary term,   by Lemma \ref{003},  we have \be\label{qp59}\ba \max_{x\in \overline\Omega} \left| \int_{\partial\Omega}\frac{\partial \widetilde{N}}{\partial n}(x, y)F(y)dS_y\right|\le C\|F\|_{H^1},\ea\ee
which together with \eqref{311} and \eqref{qp40} yields

\begin{equation}\label{009}
\ba &\int_0^T
\max_{x\in \overline\Omega}  \left| \int_{\partial\Omega}\frac{\partial \widetilde{N}}{\partial n}(x, y)F(y)dS_y\right|dt\\ &\leq  C\int_0^T \|F\|_{H^1 }dt\\
&\leq C\int_0^T\big(R_T^{\frac{1}{2}}A_2+A_1\big)dt\\
&\leq CR_T^{\frac{1}{2}}\int_0^T\left(\frac{A_2^2}{ A_1^2}\right)^{\frac{1}{2}}A_1 dt+ C\\
&\leq CR_T^{\frac{1}{2}}\left(\int_0^T\frac{A_2^2}{ A_1^2}dt\right)^{\frac{1}{2}} \left(\int_0^T A_1^2 dt\right)^{\frac{1}{2}}+
C \\
&\leq C (\varepsilon)R_T^{1+\varepsilon}.
\ea
\end{equation}

 Then,   we use the mass equation and the boundary condition $u\cdot {n}|_{\partial \Omega}=0$ to rewrite the first term on the righthand side of \eqref{qp11} as follows:
\begin{equation}\label{369}
\ba& -\int_\Omega \nabla_y\widetilde{N}(x, y)\cdot\rho\dot{u}(y)dy\\
&=-\int_\Omega \nabla_y\widetilde{N}(x, y)\cdot\bigg(\frac{\partial(\rho u)}{\partial t}+\mathrm{div}(\rho u\otimes u)\bigg)dy\\
&=-\frac{\partial}{\partial t}\int_\Omega \partial_{y_j}\widetilde{N}(x, y)\rho u_j(y)dy+\int_\Omega \partial_{y_i}\partial_{y_j}\widetilde{N}(x, y)\rho u_iu_j(y)dy \\&
= -(\frac{\partial}{\partial t}+u\cdot\nabla)\int_\Omega \partial_{y_j}\widetilde{N}(x, y)\rho u_j(y)dy+J,
\ea
\end{equation}
 with \be\label{qw01}
J\triangleq\int_\Omega \left[\partial_{x_i}\partial_{y_j}\widetilde{N}(x, y)u_i(x)+\partial_{y_i}\partial_{y_j}\widetilde{N}(x, y)u_i(y)\right]\rho u_j(y)dy.\ee

Next, we have the following crucial  point-wise estimate on $J.$
\begin{proposition}\label{qp08} For $J$   as in \eqref{qw01}, there exists a generic positive constant $C(\Omega)$ such that for any $x\in \Omega$ with $\varphi(x)\not=0, $
\begin{equation}\label{qp379}\ba
|J(x)|\leq& C(\Omega) \int_\Omega\frac{\rho|u|^2(y)}{|x-y|}dy + C(\Omega)\int_\Omega\frac{|u(x)-u(y)|}{|x-y|^2}\rho|u|(y)dy \\& + C(\Omega)\int_\Omega\frac{|u(x')-u(y)|} {|x'-y|^2}\rho|u|(y)dy , \ea
\end{equation} where  \be   \label{qpx01}  x'\triangleq \varphi^{-1} \left( \frac{\varphi(x)}{|\varphi(x)|}\right)\in \partial\Omega . \ee
\end{proposition}

 {\it Proof.}  Using \eqref{348}, we  rewrite $J$ as
\begin{equation}\label{qw10}
\begin{split}
J(x)
=&\int_\Omega \partial_{x_i}\partial_{y_j}\widetilde{N}(x, y)[u_i(x)-u_i(y)]\rho u_j(y)dy\\
 &-\frac{1}{2\pi}\int_\Omega \Lambda_{i, j}(\varphi(y), \varphi(x))\rho u_iu_j(y)dy\\
 &-\frac{1}{2\pi}\int_\Omega \Lambda_{i, j}(\varphi(y), w(x))\rho u_iu_j(y)dy
\triangleq \sum_{l=1}^3J_l(x),
\end{split}
\end{equation} with \be\label{qw11}
\Lambda_{i, j}(\varphi(y), v(x)) \triangleq(\partial_{x_i} \partial_{y_j}+\partial_{y_i}\partial_{y_j}) \mathrm{log}\left| \varphi(y)-v(x) \right|,  w(x)\triangleq\frac{\varphi(x)}{|\varphi(x)|^2}.  \ee

Thus, for $J_1(x),$ direct computations yield that
\begin{equation}\label{qp07}
|J_1(x)|\leq C(\Omega)\int_\Omega\frac{|u(x)-u(y)|}{|x-y|^2}\rho|u|(y)dy.
\end{equation}

Then,  to estimate $J_2(x) $ and $ J_3(x),$ we check that for    $v(x) \in \{\varphi(x),  w(x)\}, $
\begin{equation}\label{384}
\begin{split}&
\Lambda_{i, j}(\varphi(y), v(x)) \\ &= \partial_{y_j}\left(\partial_{x_i}\mathrm{log}|v(x)-\varphi(y)| +\partial_{y_i}\mathrm{log}|v(x)-\varphi(y)|\right)\\
&= \partial_{y_j}\bigg[\frac{(v_k(x)-\varphi_k(y)) (\partial_iv_k(x)-\partial_i\varphi_k(y))}{|v(x)-\varphi(y)|^2}\bigg] \\&
=  \frac{ (\varphi_k(y)-v_k(x))\partial_j\partial_i\varphi_k(y)}{|v(x) -\varphi(y)|^2}+ \frac{\partial_j\varphi_k(y)(\partial_i\varphi_k(y)-\partial_iv_k(x)) }{|v(x) -\varphi(y)|^2} \\
&\quad+2\frac{(v_k(x)-\varphi_k(y))(\partial_iv_k(x)-\partial_i \varphi_k(y))( \varphi_s(y)-v_s(x))\partial_j\varphi_s(y)} {|v(x)-\varphi(y)|^4}.
\end{split}
\end{equation}

Thus,  for $J_2(x)$,  it follows from \eqref{3353}--\eqref{3355}  and \eqref{384} that \bnn |\Lambda_{i, j}(\varphi(y), \varphi(x))|\le C(|\Omega|)|x-y|^{-1}, \enn which in particular implies
\begin{equation} \label{qp06}
\begin{split}
|J_2(x)|
\leq C(\Omega)\int_\Omega\frac{\rho|u|^2(y)}{|x-y|}dy.
\end{split}
\end{equation}

FInally, it remains to estimate the more difficult term $J_3. $ According to \eqref{3353}, we have
\be\label{33387}\ba \left|\Lambda_{i, j}(\varphi(y), w(x))u_i(y)\right|
&\le \frac{C|u|}{ |\varphi(y)-w(x)|}+\frac{C |(\partial_i w_k(x)-\partial_i \varphi_k(y))u_i(y)|}{ |\varphi(y)-w(x)|^2}.\ea\ee

Now, the central issue is to estimate $(\partial_iw_k(x)-\partial_i\varphi_k(y))u_i(y). $

First, for $\varphi(x),  \varphi(y)\in \mathbb{D}$ with $\varphi(x)\not=0$,  we have
\begin{equation}\label{3387}
\left|\varphi(y)-\frac{\varphi(x)}{|\varphi(x)|}\right| \leq\left|\varphi(y)-w(x)\right|, \, \,
\left|\varphi(y)-\varphi(x)\right|\leq\big|\varphi(y)- w(x) \big|.
\end{equation}  A direct consequence of \eqref{3387} shows that   for $x, y\in  \Omega$ with  $\varphi(x)\not=0, $
\begin{equation}\label{3356}
\begin{split}
\left| \varphi(y)-w(x) \right| \ge  1-|\varphi(x)| ,  \end{split}
\end{equation} which implies   that $$ |\varphi(y)-w(x)|^{-1}\le (1-|\varphi(x)|)^{-1}\le 4, $$ provided $|\varphi(x)|\le 3/4. $ Thus,  from now on,  we always assume that 
\be\label{qp02} |\varphi(x)|>3/4. \ee

Then,  it follows from      \eqref{3353} and  \eqref{3387}--\eqref{qp02} that \be\label{3394}\ba \left|\frac{\partial_{x_i}\varphi_k(x)}{|\varphi(x)|^2}
-\partial_{y_i}\varphi_k(y)\right|&=  \left| \partial_{x_i}\varphi_k(x)-\partial_{y_i}\varphi_k(y)
+\partial_{x_i}\varphi_k(x)\frac{1-|\varphi(x)|^2}{|\varphi(x)|^2}\right| \\& \le C|x-y|+C(1-|\varphi(x)|) \\& \le C|\varphi(y)-w(x)|.   \ea\ee
   Direct computation yields that for $w$     as in  \eqref{qw11}
\be\label{3395}  \ba  \partial_{x_i}w_k(x)-\partial_{y_i}\varphi_k(y)&=\frac{\partial_{x_i}\varphi_k(x)}{|\varphi(x)|^2} -\frac{2\varphi_k(x)\varphi_l(x)\partial_{x_i}\varphi_l(x)}{|\varphi(x)|^4} -\partial_{y_i}\varphi_k(y) \\&=\tilde K_{i, k}  -2\varphi_k(x')\varphi_l(x')\partial_{y_i}\varphi_l(y),  \ea\ee
where $x'$ is as in    \eqref{qpx01} and
\bnn\ba \tilde K_{i, k} \triangleq \frac{\partial_{x_i}\varphi_k(x)}{|\varphi(x)|^2}
-\partial_{y_i}\varphi_k(y)-2\varphi_k(x')\varphi_l(x')\left( \frac{\partial_{x_i}\varphi_l(x)}{|\varphi(x)|^2}-\partial_{y_i}\varphi_l(y) \right).
\ea
\enn   It follows from \eqref{3394} that  $\tilde K_{i, k}$ satisfies
   \bnn |\tilde K_{i, k}|\le C|\varphi(y)-w(x)|, \enn   which together with \eqref{3395} implies \be \label{qp03}|(\partial_i w_k(x)-\partial_i \varphi_k(y))u_i(y)|\le C|\varphi(y)-w(x)||u|+C| \varphi_l(x^\prime) \partial_{y_i}\varphi_l(y)u_i(y)| .\ee
   Since $x'\in \partial \Omega $ and $u\cdot n=0$ on $\partial \Omega, $ we have by \eqref{1}\bnn \varphi_l( x^\prime) \partial_{y_i}\varphi_l(x^\prime) u_i(x^\prime)=0, \enn which yields that
\be \label{qp04}\ba  &| \varphi_l(x^\prime) \partial_{y_i}\varphi_l(y)u_i(y)| \\& =| \varphi_l( x^\prime)(\partial_{y_i}\varphi_l(y)-\partial_{y_i}\varphi_l(x^\prime ))u_i(y)+ \varphi_l( x^\prime) \partial_{y_i}\varphi_l(x^\prime) (u_i(y)-u_i(x^\prime))|\\& \le C|y-x^\prime||u|+C|u(y)-u(x^\prime)|\\& \le C|\varphi(y)-w(x)||u|+C|u(y)-u(x^\prime)|
\ea
\ee
where in the last inequality we have used \be\label{3398} |y-x^\prime|\le C|\varphi(y)-\varphi(x^\prime)|
\le C |\varphi(y)-w(x)|, \ee
 which comes from \eqref{3355} and \eqref{3387}.

Consequently,  combining \eqref{qp03} with \eqref{qp04} gives
\begin{equation}\nonumber
\begin{split}
|\big(\partial_{x_i}w_k(x)-\partial_{y_i}\varphi_k(y)\big)u_i(y)|  \le C|\varphi(y)-w(x)||u|+C|u(y)-u(x^\prime)|,
\end{split}
\end{equation}
which together with  \eqref{33387} yields
\bnn\ba \left|\Lambda_{i, j}(\varphi(y), w(x))u_i(y)\right|&\le \frac{C|u|}{ |\varphi(y)-w(x)|}+\frac{C |u(y)-u(x^\prime)|}{ |\varphi(y)-w(x)|^2}\\ &\le \frac{C|u|}{ |y-x|}+\frac{C |u(y)-u(x^\prime)|}{ |y-x'|^2}, \ea\enn
where in the last inequality we have used \eqref{3398}.  As a direct consequence,  we  arrive at
\begin{equation*}
\begin{split}
|J_3(x)| \le C\int_\Omega\frac{\rho|u|^2}{ |y-x|}dy+C\int_\Omega\frac{ |u(y)-u(x^\prime)|}{ |y-x'|^2}\rho |u|dy,
\end{split}
\end{equation*}
which together with \eqref{qp06} and \eqref{qp07} gives \eqref{qp379} and
  finishes the    proof of Proposition \ref{qp08}. \thatsall

\begin{remark}\label{qw31}
In particular, when $\Omega$ is the unit disc $\mathbb{D}$ itself, the computation above can be greatly simplified, once we set $\varphi$ is the identity.
Actually such reduced case is exactly the starting point of our proof.
\end{remark}

\begin{remark}\label{qw32} We call the first term on the right hand side of \eqref{qp379} having singularity of order 1(or simply ``of order 1" in abbreviation),  and the last two terms of commutator type. Indeed,
the usual  Caldron-type commutator  takes the form of $[u, R_iR_j](\rho u)$(where $R_i$ is the usual Riesz transformation). Writing such singular integral operator in integral form, we check that formally it looks like
\begin{equation*}
[u, R_iR_j](\rho u)(x)=\int_{\mathbb{R}^2}\frac{u(x)-u(y)}{|x-y|^2}\rho u(y)dy.
\end{equation*}
Such representation coincides with the terms of commutator type mentioned above except for the integral domains.\end{remark}

\begin{remark}
It should be mentioned here  that the situation is much more different from \cite{huang2016existence} in which the boundary is assumed to be periodic.
The setting on periodic case is equivalent with dealing the compressible  Navier-Stokes system \eqref{11} on the torus of dimension 2. Such manifold is compact orientable and out of boundary. Consequently, the usual commutator theory can be applied directly in
\cite{huang2016existence} to get a proper control of $F$ (see \cite[ (3.34)]{huang2016existence} for example).

But in this paper, we consider the domain with boundary which is essentially different from the periodic case. Such
difference is reflected in the technical difficulties that classical commutator theory is no longer available. The main purpose of
Proposition \ref{qp08} is to get over it. In other words, we are trying to establish a suitable commutator theory for
bounded domains, and the approach we adapted is the combination of Green's functions in a disc with conformal mapping.

Let us look closer at \eqref{qp379} which consists of three terms. As mentioned in Remark \ref{qw32}, we note that the second one
$$
\int_\Omega\frac{|u(x)-u(y)|}{|x-y|^2}\rho|u|(y)dy
$$
corresponds to the usual commutator just like \cite[(2.7)]{huang2016existence}. Thus, we remark that the rest two terms
represent the difference between the periodic domains and general bounded ones and  can be regarded as the ``error" terms we must face during  treating the  general bounded domain case. 
In particular, the behavior of
$$
\int_\Omega\frac{|u(x')-u(y)|} {|x'-y|^2}\rho|u|(y)dy
$$ is indeed due to the effect of the boundary and the Navier-slip conditions which
  in fact coincides with usual commutator except that $x'$ is on the boundary, which can be regarded as the commutator of the boundary.
\end{remark}


Now, to derive the precise control of $F$, we focus on commutator type terms which are of vital importance. Although such terms are not ``real" commutator, and the classical $L^p$ theory doesn't apply directly, we still have alternative method to obtain proper estimates. 
\begin{proposition}\label{006}
For any $\varepsilon>0, $  we have
\begin{equation}\label{qp32}
\int_{0}^{T}\max_{x\in\overline\Omega} \int_\Omega
 \frac{|u(x )-u(y)|}{|x-y|^2} \rho|u|(y)dy dt\leq C(\varepsilon )R_T^{1+\frac{\beta}{4}+3\varepsilon}.
\end{equation}
\end{proposition}
\begin{proof}
First,  for $2<p<4$ which will be determined later,
by Sobolev's embedding theorem (Theorem 5 of \cite[Chapter 5]{2010Partial}),   we have for any $x, y\in \overline{\Omega}, $
\begin{equation*}
|u(x)-u(y)|\leq C(p )\|\nabla u\|_{L^p}|x-y|^{1-\frac{2}{p}},
\end{equation*}
which implies
\begin{equation}\label{qp30}
\ba
\int_\Omega\frac{|u(x )-u(y)|}{|x-y|^2}\rho|u|(y)dy  &\leq C(p) \int_\Omega\frac{\|\nabla u\|_{L^p}\cdot|x-y|^{1-\frac{2}{p}}}{|x-y|^2}\rho|u|(y)dy\\
&=C(p)\|\nabla u\|_{L^p}\int_\Omega |x-y|^{-(1+\frac{2}{p})} \rho|u|(y)dy.
\ea
\end{equation}

Then, for $\delta>0 $ and $\varepsilon_0\in (0, (p-2)/8)  $   which will be determined later,  on the one  hand,  we use  \eqref{25}  to get
\begin{equation}\label{qp35}
\ba
 &
\int_{|x-y|<2\delta}|x-y|^{-\left(1+\frac{2}{p}\right)}\rho|u|(y)dy\\
&\leq C(p)R_T
\left(\int_{|x-y|<2\delta} |x-y|^{-\left(1+\frac{2}{p}\right) (1+\varepsilon_0)} dy\right)^{\frac{1}{1+\varepsilon_0}}
\|u\|_{L^{\frac{1+\varepsilon_0}{\varepsilon_0}}}\\
&\leq C(p)  R_T
\delta^{1-\frac{2}{p}-\frac{2\varepsilon_0}{1+\varepsilon_0}} \left(\frac{1+\varepsilon_0}{\varepsilon_0}\right)^{\frac{1}{2}}\|u\|_{H^1}\\
&\leq C(p)   R_T \varepsilon_0^{-\frac{1}{2}} A_1
\delta^{1-\frac{2}{p}-\frac{2\varepsilon_0}{1+\varepsilon_0}} .
\ea
\end{equation}
  On the other hand,  for $\nu=R_T^{-\frac{\beta}{2}}\nu_{0}$ as in \eqref{nu1q},  we use  Lemma \ref{0001} to derive
\begin{equation}
\ba
  &\int_{|x-y|>\delta}|x-y|^{-\left(1+\frac{2}{p}\right)}\rho|u|(y)dy\\
&\leq
C(p)
\left(\int_{|x-y|>\delta} |x-y|^{-\left(1+\frac{2}{p}\right) (\frac{2+\nu}{1+\nu})} dy\right)^{\frac{1+\nu}{2+\nu}}
\left(\int_\Omega\rho^{2+\nu}|u|^{2+\nu}dx\right)^{\frac{1}{2+\nu}}\\
&\leq C(p)R_T^{ \frac{1+\nu}{2+\nu}}
 \delta^{-\frac{2}{p}+\frac{\nu}{2+\nu}}.
\ea
\end{equation}

Now, we choose $\delta >0$ such that \be \delta^{-\frac{2}{p}+\frac{\nu}{2+\nu}} =
A_1^{ \frac{2}{p} },\ee which in particular implies \be \label{qp36} A_1\delta^{1-\frac{2}{p}-\frac{2\varepsilon_0}{1+\varepsilon_0}} =A_1^{ \frac{2}{p} },\ee
 provided we set \be\label{iqu1} \varepsilon_0=\frac{(p-2)\nu}{8+(6-p) \nu}\in \left(0, \frac{p-2}{8}\right). \ee

Then, it follows from \eqref{qp35}--\eqref{qp36} that
\begin{equation*}
\ba\label{3107}&
 \int_\Omega |x-y|^{-\left(1+\frac{2}{p}\right)}\rho|u|(y)dy \\&\le \left( \int_{|x-y|<2\delta} +\int_{|x-y|> \delta}\right)|x-y|^{-\left(1+\frac{2}{p}\right)}\rho|u|(y)dy \\
&\leq C(p)  R_T\varepsilon_0^{-1/2}   A_1^{\frac{2}{p}}+C(p)R_T^{ \frac{1+\nu}{2+\nu}}A_1^{\frac{2}{p}}\\
&\leq C(p)  R_T^{1+\beta/4}   A_1^{\frac{2}{p}} ,
\ea
\end{equation*}
where in the last line we have used $\varepsilon_0^{-1/2}\le C(p)\nu^{-1/2}\le C(p)R_T^{\beta/4} $ due to \eqref{iqu1}. Combining this,  \eqref{qp30},  and
 \eqref{qp31} shows that for any $\varepsilon>0,$
\begin{equation*}
\ba
  &
\int_\Omega\frac{|u(x )-u(y)|}{|x-y|^2}\rho|u|(y)dy \\
&\leq
C(p , \varepsilon )R_T^{\frac{3}{2}-\frac{1}{p}+\varepsilon+\frac{\beta}{4}}A_1^{1+\frac{2}{p}} \left(\frac{A_2^2}{ A_1^2}\right)^{\frac{1}{2}-\frac{1}{p}}
+C(p, \varepsilon)R_T^{1+\varepsilon+\frac{\beta}{4}}A_1^{1+\frac{2}{p}}.
\ea
\end{equation*}

Finally, integrating this  with respect to $t$ and using the H\"{o}lder inequality,  we arrive at
\begin{equation*} \label{3108}
\begin{split}
    &\int_{0}^{T}\max_{x\in\overline\Omega}\int_\Omega\frac{|u(x(t))-u(y)|}{|x-y|^2}\rho|u|(y)dy dt\\
&\leq C(p, \varepsilon)R_T^{\frac{3}{2}-\frac{1}{p}+\varepsilon +\frac{\beta}{4}}
\left(\int_0^TA_1^2dt\right)^{\frac{1}{2}+\frac{1}{p}} \left(\int_0^T\frac{A_2^2}{ A_1^2}dt \right)^{\frac{1}{2}-\frac{1}{p}}\\
&\quad+C(p,\varepsilon)R_T^{1+\varepsilon+\frac{\beta}{4}}
\left(\int_0^T A_1^2 dt\right)^{\frac{1}{2}+\frac{1}{p}}\\
&\leq C(p,\varepsilon)R_T^{2-\frac{2}{p}+\frac{\beta}{4}+2\varepsilon},
\end{split}
\end{equation*}
where in the last line we have used  \eqref{qp87} and   \eqref{qp40}.  This,  after choosing $p=2/(1-\varepsilon), $  in particular yields \eqref{qp32} and finishes the proof of Proposition \ref{006}.
\end{proof}

 Now we are in a position to obtain the upper bound of the density which plays an essential role in the whole procedure.
\begin{proposition} \label{0002} Assume that \eqref{17} holds. Then   there exists some positive constant  $C  $ depending only on $\Omega$,  $T,  \mu,  \beta,$ $ \gamma,     \|\rho_0\|_{L^\infty},$ and $\| u_0\|_{H^1} $ such that
\begin{equation}\label{3103}
\sup_{0\leq t\leq T}(\|\rho\|_{L^{\infty}}+\|u\|_{H^{1}})+\int_{0}^{T}\left( \|\omega\|_{H^{1}}^{2}+\|F\|_{H^{1}}^{2}+\|\sqrt\rho\dot{u}\|_{L^2}^2 \right)dt\leq C .
\end{equation}

\end{proposition}
\begin{proof}
First, for   $\theta(\rho) $ as in \eqref{qis1}, we have by \eqref{11}, \eqref{303},   and \eqref{ou1r}
\begin{equation*}
\ba
  \frac{D}{Dt}\theta(\rho) =-(2\mu+\lambda)\mathrm{div} u \le -F
\ea
\end{equation*}
which together with \eqref{qp11}, \eqref{369},  \eqref{qp59}, and \eqref{qp379} gives
\be\label{qp66}\ba   \frac{D}{Dt}\theta(\rho)
&\le \frac{D}{Dt}\int_\Omega \partial_{y_j}\widetilde{N}(x, y)\rho u_j(y)dy+C\|F\|_{H^1} +|J|\\ &\le \frac{D}{Dt}\int_\Omega \partial_{y_j}\widetilde{N}(x, y)\rho u_j(y)dy+C\|F\|_{H^1} +\max_{x\in \overline \Omega}\int_\Omega\frac{\rho |u|^2}{|x-y|}dy\\&\quad+\max_{x\in \overline \Omega}\int_\Omega\frac{ |u(x)-u(y)|}{|x-y|^2}\rho|u|(y)dy. \ea\ee

Then, on the one hand, direct computation yields that for $\nu=R_T^{-\frac{\beta}{2}}\nu_{0}$ as in Lemma \ref{0001},
\be\label{qp68}\ba &\left|\int_\Omega \partial_{y_j}\widetilde{N}(x, y)\rho u_j(y)dy\right|\\ &\le C\int_\Omega |x-y|^{-1}\rho(y)|u(y)|dy\\&\le C\left(\int_\Omega|x-y|^{-\frac{2+\nu}{1+\nu}} dy\right)^{\frac{1+\nu}{2+\nu}} \left(\int_\Omega\rho^{2+\nu}|u|^{2+\nu}dy\right)^{\frac{1 }{2+\nu}}\\&\le C\nu^{-\frac{1+\nu}{2+\nu}}R_T^{\frac{1+\nu}{2+\nu}}\left(\int_\Omega\rho |u|^{2+\nu}dy\right)^{\frac{1 }{2+\nu}} \\&\le C R_T^{\left(1+\frac{\beta}{2}\right)\frac{1+\nu}{2+\nu}}\\&\le C R_T^{\frac{2+\beta}{3}}  , \ea\ee
where in the fourth inequality we have used  \eqref{qp90}.

On the other hand, Sobolev's inequality leads to
\begin{equation*}
\begin{split}
 \int_\Omega\frac{\rho |u|^2(y)}{|x-y|}dy &\leq R_T \left(\int_\Omega |x-y|^{-\frac{3}{2}}dy \right)^{\frac{2}{3}}
\left(\int_\Omega|u|^6 dy\right)^{\frac{1}{3}} \\
&\leq CR_T \|u\|_{H^1}^2,
\end{split}
\end{equation*}
which together with \eqref{310} and \eqref{30} shows
\begin{equation}\label{0008}
\begin{split}
\int_0^T\max_{x\in \overline\Omega}\int_\Omega\frac{\rho |u|^2(y)}{|x-y|}dy dt
 \leq CR_T.
\end{split}
\end{equation}

Finally, integrating \eqref{qp66} with respect to $t,$ we obtain after using   \eqref{009}, \eqref{qp32}, \eqref{qp68}, and \eqref{0008} that \bnn R_T^\beta\le C(\varepsilon) R_T^{\max\left\{1+\frac{\beta}{4}+3\varepsilon, {\frac{2+\beta}{3}}\right\}} .\enn  Since $\beta>4/3,$ this  in particular implies
\begin{equation}\label{qp70}
\sup_{0\leq t\leq T}\|\rho\|_{L^\infty}\leq C ,
\end{equation} which together with
 \eqref{312},  \eqref{qp87},   \eqref{311}, and Gronwall's inequality gives \eqref{3103} and finishes the proof of Proposition \ref{0002}.
\end{proof}

\section{\label{sec3}A priori estimates (II): lower and higher order ones}
Once obtaining  the upper bound  of the density, we will  proceed to study the lower and high order estimates which are indeed quite similar to those of \cite{huang2016existence},  except that the boundary terms do bring some trouble that requires further considerations.
We mainly focus on the boundary terms occurring along the way to our final goal,  which is different from standard process. We mainly borrow some ideas from \cite{huang2016existence,caili01}.

First, following \cite{caili01}, we give a Poincar\'{e} type estimate to treat the effect of boundary conditions where the construction depends on slip condition $u\cdot n|_{\partial\Omega} = 0$.
\begin{lemma} For $p\ge 1,$
there exist   positive constants $C_1(p,\Omega)$ and  $C_2(\Omega)$ such that
\begin{equation}
\|\dot{u}\|_{L^p}\leq C_1(\|\nabla\dot{u}\|_{L^2} + \|\nabla u\|^2_{L^2} ), \label{41}
\end{equation}
\begin{equation}
\|\nabla\dot{u}\|_{L^2} \leq C_2(\|\mathrm{div}\dot{u}\|_{L^2} + \|\mathrm{curl}\dot{u}\|_{L^2} + \|\nabla u\|_{L^4}^2 ).  \label{402}
\end{equation}
\end{lemma}
\begin{proof} First, we extend $n$ to the whole domain $\Omega$ smoothly. The approach is not unique, and we fix one through out the paper.

 Then, for   $n^\bot \triangleq (n_2, -n_1) $ as in \eqref{q1w}, we have by $u\cdot n|_{\partial\Omega} = 0$, \be\label{qe03} u=(u\cdot n^\bot)n^\bot  \mbox{  on  }\partial\Omega, \ee and
\begin{equation*}\ba
\dot{u}\cdot n&=(u\cdot\nabla)u\cdot n=-(u\cdot\nabla)n\cdot u\\&=-(u\cdot n^\bot)(u\cdot\nabla)n\cdot n^\bot= (u\cdot n^\bot)(u\cdot\nabla)n^\bot\cdot n  \mbox{  on  }\partial\Omega ,\ea
\end{equation*} which gives\be\label{qe01} (\dot{u}-(u\cdot n^\bot)(u\cdot\nabla)n^\bot)\cdot n=0.\ee It follows from \eqref{qe01} and Poincar\'{e}'s inequality that
\begin{equation*}
\|\dot{u}-(u\cdot n^\bot)(u\cdot\nabla)n^\bot\|_{L^{3/2}}\leq C\|\nabla(\dot{u}-(u\cdot n^\bot)(u\cdot\nabla)n^\bot)\|_{L^{3/2}},
\end{equation*}
which leads to
\begin{equation*}
\|\dot{u}\|_{L^{3/2}}\leq C(\|\nabla\dot{u}\|_{L^{3/2}}+\|\nabla u\|_{L^2}^2).
\end{equation*}
Combining this and the  Sobolev embedding theorem implies that for $p>1,$
\begin{equation*}\ba
\|\dot{u}\|_{L^{p}}&\leq C(p)(\|\dot{u}\|_{L^2}+\|\nabla\dot{u}\|_{L^2})\\& \le C(p)(\|\dot{u}\|_{L^{3/2}}+\|\nabla\dot{u}\|_{L^{3/2}})+C (p) \|\nabla\dot{u}\|_{L^2}\\&\leq C(p)(\|\nabla\dot{u}\|_{L^2}+\|\nabla u\|_{L^2}^2),\ea
\end{equation*} which proves \eqref{41}.

 Finally, \eqref{402} is a direct consequence of    \eqref{qe01} and  \eqref{24}.
\end{proof}

Now we are ready to derive the lower order a priori estimates step by step.
\begin{lemma}\label{042}   There is a positive constant $C $ depending only on $\Omega , T,  \mu,\beta, \gamma$, $\|\rho_0\|_{L^\infty},$ and $\|u_0\|_{H^1}$ such that
\begin{equation}
\sup_{0\leq t\leq T}\sigma\int_\Omega\rho|\dot{u}|^2dx + \int^T_0 \sigma\|\nabla\dot{u}\|_{L^2}^2dt \leq C ,  \label{410}
\end{equation}
with $\sigma(t)=\min{\{1, t\}}$.  Moreover,  for any $p\in[1, \infty)$,  there is a positive constant $C(p )$ depending only on $p$,   $\Omega$, $T$,  $\mu$,  $\beta$,  $\gamma$,  $\|\rho_0\|_{L^\infty},$ and $\|u_0\|_{H^1}$ such that
\begin{equation}
\sup_{0\leq t\leq T} \sigma\|\nabla u\|^2_{L^p}\leq C(p ).  \label{411}
\end{equation}
\end{lemma}
\begin{proof}
First,  operating $\dot{u}_j[\partial/\partial t + \mathrm{div}(u\cdot)]$ to $\eqref{11}^j_2$,
summing with respect to $j$,   integrating the resulting equation over $\Omega$, and using the boundary condition \eqref{15}, one gets after integration
by parts that
\begin{equation}
\begin{split}
\frac{1}{2}\frac{d}{dt}\int_\Omega\rho|\dot{u}|^2dx =&\int_\Omega(\dot{u}\cdot\nabla F_t +\dot{u}_j\mathrm{div}(\partial_jF u))dx\\
&+\mu\int_\Omega(\dot{u}\cdot\nabla^\bot\omega_t +\dot{u}_j\partial_k((\nabla^\bot\omega)_ju_k))dx\\ \label{413}
=&\tilde J_1+\tilde J_2.
\end{split}
\end{equation}

Now,  we will estimate $\tilde J_1$ and $\tilde J_2$ respectively.
First, using the boundary condition \eqref{15}, we obtain after integration by parts that
\begin{equation}
\begin{split}
\tilde J_1=&\int_\Omega(\dot{u}\cdot\nabla F_t +\dot{u}_j\mathrm{div}(\partial_jF u))dx\\
=&\int_\Omega(\dot{u}\cdot\nabla F_t +\dot{u}_j\partial_j(u\cdot\nabla F)+\dot{u}_j\partial_jF\mathrm{div}u+\dot{u}\cdot\nabla u\cdot\nabla F)dx\\
=&\int_{\partial\Omega}(F_t+u\cdot\nabla F)(\dot{u}\cdot n)ds-\int_\Omega(F_t+u\cdot\nabla F)\mathrm{div}\dot{u}dx\\
&+\int_\Omega(\dot{u}_j\partial_jF\mathrm{div}u+\dot{u}\cdot\nabla u\cdot\nabla F)dx\\
=&\int_{\partial\Omega} F_t (\dot{u}\cdot n)ds+\int_{\partial\Omega} u\cdot\nabla F (\dot{u}\cdot n)ds-\int_\Omega(2\mu+\lambda)(\mathrm{div}\dot{u})^2dx \\
&+\int_\Omega\lambda'(\rho)\rho(\mathrm{div}u)^2 \mathrm{div}\dot{u}dx+\int_\Omega(2\mu+\lambda) \partial_iu_j\partial_ju_i\mathrm{div}\dot{u} dx\\
&+\gamma\int_\Omega P\mathrm{div}u\mathrm{div}\dot{u} dx+\int_\Omega(\dot{u}_j\partial_jF\mathrm{div}u+\dot{u}\cdot\nabla u\cdot\nabla F)dx \\
\le &\int_{\partial\Omega} F_t (\dot{u}\cdot n)ds+\int_{\partial\Omega} u\cdot\nabla F (\dot{u}\cdot n)ds-\mu\int_\Omega (\mathrm{div}\dot{u})^2dx \\
&+C+C\|\nabla u\|_{L^4}^4+C \int |\dot{u}| |\nabla F||\nabla u|dx ,  \label{414}
\end{split}
\end{equation}
where in the third equality we have used the following fact
\begin{equation*}
\begin{split}
F_t+u\cdot\nabla F=&\lambda_t\mathrm{div}u+(2\mu+\lambda)\mathrm{div}u_t +u\cdot\nabla((2\mu+\lambda)\mathrm{div}u)+P_t+u\cdot\nabla P \\
=&(\lambda_t+u\cdot\nabla\lambda)\mathrm{div}u +(2\mu+\lambda)\mathrm{div}\dot{u}-(2\mu+\lambda)\mathrm{div}(u\cdot\nabla u)\\
&+(2\mu+\lambda)\nabla \mathrm{div}u+\gamma P\mathrm{div}u \\
=&-\rho\lambda'(\rho)(\mathrm{div}u)^2+(2\mu+\lambda)\mathrm{div}\dot{u} +(2\mu+\lambda)\partial_iu_j\partial_ju_i +\gamma P\mathrm{div}u.
\end{split}
\end{equation*}
For the first   term on the righthand side of \eqref{414}, we have
\begin{equation}\label{qyi01}
\begin{split}
&\int_{\partial\Omega}F_t(\dot{u}\cdot n)ds\\&
= \int_{\partial\Omega}F_t(u\cdot\nabla u\cdot n)ds \\&
=  -\frac{d}{dt}\int_{\partial\Omega}F(u\cdot\nabla n\cdot u)ds+\int_{\partial\Omega}F(u\cdot\nabla n\cdot u)_t ds\\&
= -\frac{d}{dt}\int_{\partial\Omega}F(u\cdot\nabla n\cdot u)ds+\int_{\partial\Omega}F(u_t\cdot\nabla n\cdot u)ds+\int_{\partial\Omega}F(u\cdot\nabla n\cdot u_t)ds\\&
= -\frac{d}{dt}\int_{\partial\Omega}F(u\cdot\nabla n\cdot u)ds+\left[\int_{\partial\Omega}F(\dot{u}\cdot\nabla n\cdot u)ds+\int_{\partial\Omega}F(u\cdot\nabla n\cdot\dot{u})ds\right]\\
&\quad- \int_{\partial\Omega}F((u\cdot\nabla u)\cdot\nabla n\cdot u)ds-\int_{\partial\Omega}F(u\cdot\nabla n\cdot(u\cdot\nabla u))ds \\&
= -\frac{d}{dt}\int_{\partial\Omega}F(u\cdot\nabla n\cdot u)ds+K_1+K_2+K_3.
\end{split}
\end{equation}

Then,  by \eqref{311},  \eqref{305},     \eqref{41},  \eqref{qp70} and \eqref{3103}, we get
\begin{equation}\label{qyi02}
\begin{split}
K_1=&\int_{\partial\Omega}F(\dot{u}\cdot\nabla n\cdot u)ds+\int_{\partial\Omega}F(u\cdot\nabla n\cdot\dot{u})ds\\
\leq&C\| u\|_{H^1}\| \dot{u}\|_{H^1}\|  F\|_{H^1}\\
\leq&C \|\nabla u\|_{L^2}(\|\nabla\dot{u}\|_{L^2}+\|\nabla u\|_{L^2}^2)(\|\sqrt\rho\dot{u}\|_{L^2}+\|\nabla u\|_{L^2})\\
\leq&\varepsilon\|\nabla\dot{u}\|_{L^2}^2+C(\varepsilon  ) \|\sqrt\rho\dot{u}\|_{L^2}^2+ C(\varepsilon  ).
\end{split}
\end{equation}
For $K_2,$ by taking advantage of \eqref{qe03}, we have
\begin{equation}\label{qyi03}
\begin{split}
|K_2|=&\left|\int_{\partial\Omega}F((u\cdot\nabla u)\cdot\nabla n\cdot u)ds \right|\\=&\left|\int_{\partial\Omega}F(u\cdot n^\bot)n^\bot \cdot\nabla u_i\partial_i n_j u_jds \right|\\=&\left|\int_{\Omega}\nabla^\bot \cdot\left(\nabla u_i\partial_i n_j u_jF(u\cdot n^\bot)\right)dx \right|\\=&\left|\int_{\Omega}\nabla u_i\cdot\nabla^\bot \left(\partial_i n_j u_jF(u\cdot n^\bot)\right)dx \right|\\ \le& C\int_\Omega |\nabla u|\left(|F||u|^2+|F||u||\nabla u|+|u|^2|\nabla F|\right)dx \\ \le & C\|\nabla u \|_{L^4}\left(\|F\|_{L^4}\|u\|_{L^4}^2+\|F\|_{L^4}\|u\|_{L^4}\|\nabla u\|_{L^4}+\|\nabla F\|_{L^2}\|u\|_{L^8}^2\right)\\ \le & C\|\nabla u \|_{L^4}^4+C \|F\|_{H^1}^2+C  \\
\leq&C\|\nabla u\|_{L^4}^4+C  \|\sqrt\rho\dot{u}\|_{L^2}^2+C.
\end{split}
\end{equation}
Similarly, for $K_3,$ we also have \bnn
|K_3|\leq C\|\nabla u\|_{L^4}^4+C  \|\sqrt\rho\dot{u}\|_{L^2}^2+C,\enn
which together with \eqref{qyi01}--\eqref{qyi03} leads to
\begin{equation}\label{qs2}
\begin{split}
\int_{\partial\Omega}F_t(\dot{u}\cdot n)ds
\leq&-\frac{d}{dt}\int_{\partial\Omega}F(u\cdot\nabla n\cdot u)ds+\varepsilon\|\nabla\dot{u}\|_{L^2}^2\\
&+C(\varepsilon )\left(\|\sqrt\rho\dot{u}\|_{L^2}^2+1+\|\nabla u\|_{L^4}^4\right).
\end{split}
\end{equation}
  For the second term on the righthand side of \eqref{414},
\begin{equation}\label{qs1}
\begin{split}
\int_{\partial\Omega}(u\cdot\nabla F)(\dot{u}\cdot n)ds =& \int_{\partial\Omega}(u\cdot n^\bot)n^\bot\cdot\nabla F(\dot{u}\cdot n)  ds\\
=& \int_\Omega\nabla^\bot\cdot((u\cdot n^\bot) \nabla F(\dot{u}\cdot n) )dx\\
=& \int_\Omega \nabla F\cdot\nabla^\bot((u\cdot n^\bot)(\dot{u}\cdot n) )dx\\
 \leq&C\int_\Omega |\nabla F||\dot{u}||\nabla u|dx+C\int_\Omega |\nabla F||\nabla \dot u||u|dx   \\
 \leq&C\|\nabla F\|_{L^2}\|\dot u\|_{L^4}\|\nabla u\|_{L^4} +C\|\nabla F\|_{L^2}\|\nabla\dot u\|_{L^2}\|\nabla u\|_{L^4} \\
\leq&\varepsilon \|\nabla \dot u\|_{L^2}^2 +C(\varepsilon)
\left(1+\|\sqrt\rho\dot{u}\|_{L^2}^2\right)\left(\|\nabla u\|_{L^4}^4+1\right) ,
\end{split}
\end{equation} which implies that the   last term of $\tilde J_1$ can be also bounded by the righthand side of \eqref{qs1}.

 Putting \eqref{qs2} and \eqref{qs1} into \eqref{414} leads to
\begin{equation}\label{415}
\ba
\tilde J_1\leq&-\mu\int_\Omega (\mathrm{div}\dot{u})^2dx-\frac{d}{dt}\int_{\partial\Omega}F(u\cdot\nabla n\cdot u)ds\\
&+C\varepsilon\|\nabla\dot{u}\|_{L^2}^2+C(\varepsilon)
\left(1+\|\sqrt\rho\dot{u}\|_{L^2}^2\right)\left(\|\nabla u\|_{L^4}^4+1\right) .
\ea
\end{equation}

Next we use the boundary condition \eqref{15} to get \bnn {\rm div}(u\omega)={\rm div}u \omega+u\cdot\nabla \omega=0 \mbox{ on }\partial \Omega,\enn
 which leads to
\bnn
\ba\tilde J_2=&\mu\int_\Omega(\dot{u}\cdot\nabla^\bot\omega_t +\dot{u}_j\partial_k((\nabla^\bot\omega)_ju_k))dx\\
=&-\mu\int_\Omega \mathrm{curl}\dot{u}\omega_tdx +\mu\int_\Omega\dot{u}\cdot\nabla^\bot(\mathrm{div}(u\omega)))dx \\&-\mu\int_\Omega\dot{u}_j \partial_k ((\nabla^\bot u_k)_j\omega)dx \\
=&-\mu\int_\Omega (\mathrm{curl}\dot{u})^2dx+\mu\int_\Omega \mathrm{curl}\dot{u}\mathrm{curl}(u\cdot\nabla u)dx\\&-\mu\int_\Omega {\rm curl}\dot{u}  \mathrm{div}(u\omega  )dx -\mu\int_\Omega\dot{u}_j\partial_k((\nabla^\bot u_k)_j\omega)dx \\
=&-\mu\int_\Omega (\mathrm{curl}\dot{u})^2dx+\mu\int_\Omega \mathrm{curl}\dot{u}(\nabla^\bot u)^T:\nabla udx\\
&-\mu\int_\Omega \mathrm{curl}\dot{u} \, \mathrm{div}u\,\omega dx+\mu\int_\Omega \partial_k\dot{u}_j (\nabla^\bot u_k)_j\omega dx\\
\leq&-\mu\int_\Omega (\mathrm{curl}\dot{u})^2dx+\varepsilon\|\nabla\dot{u}\|_{L^2}^2+C(\varepsilon) \|\nabla u\|_{L^4}^4.
\ea
\enn
Putting this and \eqref{415} into   \eqref{413} gives
\begin{equation}\label{qu01}
\begin{split}
&\frac{1}{2}\frac{d}{dt}\int_\Omega\rho|\dot{u}|^2dx +\mu\int_\Omega(\mathrm{div}\dot{u})^2dx+\mu\int_\Omega (\mathrm{curl}\dot{u})^2dx\\
&\leq-\frac{d}{dt}\int_{\partial\Omega}F(u\cdot\nabla n\cdot u)ds+ C\varepsilon\|\nabla\dot{u}\|_{L^2}^2\\&\quad+C(\varepsilon)
\left(1+\|\sqrt\rho\dot{u}\|_{L^2}^2\right)\left(\|\nabla u\|_{L^4}^4+1\right)\\
&\leq-\frac{d}{dt}\int_{\partial\Omega}F(u\cdot\nabla n\cdot u)ds+ C\varepsilon\|\nabla\dot{u}\|_{L^2}^2 +C(\varepsilon)
\left(1+\|\sqrt\rho\dot{u}\|_{L^2}^2\right)^2,
\end{split}
\end{equation}
where in the last inequality we have used
\begin{equation}\label{qe98}
\begin{split}
\|\nabla u\|_{L^4}^4\leq C  \|\sqrt\rho\dot{u}\|_{L^2}^2+C  ,
\end{split}
\end{equation} due to \eqref{qp31} and \eqref{3103}.
Combining \eqref{qu01}, \eqref{402},  \eqref{qe98},   and choosing $\varepsilon$ suitably small yields
\begin{equation}
\begin{split}&
\frac{d}{dt}\int_\Omega\rho|\dot{u}|^2dx+C_2^{-1}\mu\int_\Omega |\nabla\dot{u}|^2dx\\&
\leq -\frac{d}{dt}\int_{\partial\Omega}2F(u\cdot\nabla n\cdot u)ds+C \left(\|\sqrt\rho\dot{u}\|_{L^2}^2+1\right)^2. \label{43}
\end{split}
\end{equation}
Multiplying \eqref{43} by $\sigma$,  one gets \eqref{410} after using \eqref{332},   \eqref{3103}, and Gronwall's inequality.

  Finally, it follows from \eqref{qp31}, \eqref{3103}, and \eqref{410}    that for $p\geq 2$,
\begin{equation*}
\begin{split}
\sigma\|\nabla u\|_{L^p}^2\leq& C(p)\sigma\|\sqrt\rho\dot{u}\|^2_{L^2}+C(p)\leq C(p),
\end{split}
\end{equation*} which shows \eqref{411} and finishes the proof of Lemma \ref{042}.
\end{proof}

\begin{lemma}
Assume that \eqref{17} holds.  Then for any $p>2$,  there is a positive constant C depending only on $\Omega,$ $p$,  $T$,  $\mu$,  $\beta$,  $\gamma$,  $\|\rho_0\|_{L^\infty},$ and $\|u_0\|_{H^1}$ such that
\begin{equation}
\ba\label{430}
& \int_0^T( \|\nabla F\|_{L^p}+\|\nabla\omega\|_{L^p}+\|\rho\dot{u}\|_{L^p})^{1+1/p}dt +\int_0^T(\|F\|^{3/2}_{L^\infty}+\|\omega\|^{3/2}_{L^\infty})dt\\
&+\int_0^Tt(\|\nabla F\|_{L^p}^2+\|\nabla\omega\|_{L^p}^2+\|\dot{u}\|_{H^1}^2)dt\leq C,
\ea
\end{equation} and that \be\label{puw2}  \inf_{(x,t)\in \Omega\times(0,T)}\rho(x,t)\ge C^{-1}\inf_{x\in \Omega}\rho_0(x).\ee
\end{lemma}
\begin{proof} First, it follows from \eqref{41}   and \eqref{3103} that
\begin{equation*}
\begin{split}
\|\rho\dot{u}\|_{L^p} &\leq C\|\rho\dot{u}\|_{L^2}^{2(p-1)/(p^2-2)}\|\dot{u}\|_{H^1}^{p(p-2)/(p^2-2)}\\
&\leq   C\|\rho\dot{u}\|_{L^2} +C\|\rho\dot{u}\|_{L^2}^{2(p-1)/(p^2-2)}\|\nabla\dot{u}\|_{L^2}^{p(p-2)/(p^2-2)}, \\
\end{split}
\end{equation*}
which together with \eqref{3103}, \eqref{41},   \eqref{410}, and \eqref{qe98} implies that
\begin{equation}
\begin{split}
&\int_0^T\left(\|\rho\dot{u}\|_{L^p}^{1+1/p}+t\|\dot{u}\|_{H^1}^2\right)dt\\
& \leq C\int_0^T\left(\|\sqrt\rho\dot{u}\|_{L^2}^2+t\|\nabla\dot{u}\|_{L^2}^2 +t^{-1+2/(p^3-p^2-2p+2)}\right)dt\leq C.  \label{444}
\end{split}
\end{equation}
Then, the Gargliardo-Nirenberg inequality, \eqref{336}, and \eqref{3103} yield that
\begin{equation}\label{433}
\ba&
  \|F\|_{L^\infty} +  \|\omega\|_{L^\infty} \\
 &\leq C\|F\|_{L^2}+ C\|\omega\|_{L^2}+ C\|F\|_{L^2}^{\frac{p-2}{2(p-1)}}\|\nabla F\|^{\frac{p}{2(p-1)}}_{L^p}+ C\|\omega\|_{L^2}^{\frac{p-2}{2(p-1)}}\|\nabla \omega\|^{\frac{p}{2(p-1)}}_{L^p}  \\
&\leq C+C\|\rho\dot{u}\|^{\frac{p}{2(p-1)}}_{L^p} , \\
\ea
\end{equation}
which together with  \eqref{444} and \eqref{336} gives \eqref{430}.

Finally, \eqref{puw2} is a direct consequence of  \eqref{430} and  \eqref{11}$_1.$
\end{proof}

Upon now,  we have finished the lower order a priori estimates,  and will turn to the higher order ones.    Since no other
boundary terms needed to be controlled,  we just follow \cite{huang2016existence} to derive our final  a priori estimates.  For the sake of completeness,  we sketch the proof here.
\begin{proposition}\label{qou10}
Assume that \eqref{17} holds.  Then,  for $q > 2$ as in Theorem \ref{thmq1},  there is a constant $\tilde C$ depending only on $\Omega$,  $T$,  $q$,  $\mu$,  $\gamma$,  $\beta$,  $\|u_0\|_{H^1},$ and $\|\rho_0\|_{W^{1, q}}$ such that
\begin{equation}\label{434}\ba
&\sup_{0\leq t\leq T}\left(\|\rho\|_{W^{1, q}} +t\|u\|^2_{H^2}\right)\\&+\int^T_0\left(\|\nabla^2 u\|_{L^q}^{(q+1)/q}+t\|\nabla^2 u\|_{L^q}^2+t\|u_t\|^2_{H^1}\right)dt \leq \tilde C.  \ea
\end{equation}
\end{proposition}
\begin{proof}
We will follow the proof of \cite{huang2016existence}.
First,  denoting by $\Psi=(\Psi_1,  \Psi_2)$ with $\Psi_i\triangleq(2\mu + \lambda(\rho))\partial_i\rho$ $(i=1, 2)$,  one deduces from $\eqref{11}_1$ that $\Psi_i$ satisfies
\begin{equation}
\partial_t\Psi_i + (u\cdot\nabla)\Psi_i + (2\mu + \lambda(\rho))\nabla\rho\cdot \partial_iu + \rho \partial_iF + \rho\partial_iP + \Psi_i\mathrm{div}u = 0.  \label{435}
\end{equation}
For $q>2$,  multiplying \eqref{435} by $|\Psi|^{q-2}\Psi_i$ and integrating the resulting equation over $\Omega$,  we obtain after
integration by parts and using \eqref{15} to cancel out boundary term that
\begin{equation}
\frac{d}{dt}\|\Psi\|_{L^q}\leq \tilde C(1+\|\nabla u\|_{L^\infty})\|\Psi\|_{L^q}+\tilde C\|\nabla F\|_{L^q}.  \label{436}
\end{equation}

Next,    we deduce from standard $L^p$-estimate for elliptic system with boundary condition \eqref{15},  \eqref{433}, and \eqref{336} that
\begin{equation}\label{437}
\ba
\|\nabla^2u\|_{L^q}\leq& \tilde C(\|\nabla \mathrm{div}u\|_{L^q}+\|\nabla\omega\|_{L^q})\\
\leq&\tilde C(\|\nabla(2\mu+\lambda)\mathrm{div}u)\|_{L^q}+\|\mathrm{div}u\|_{L^\infty}\|\nabla\rho\|_{L^q}+\|\nabla\omega\|_{L^q})\\
\leq&\tilde C(\|\mathrm{div}u\|_{L^\infty}+1)\|\nabla\rho\|_{L^q}+\tilde C\|\nabla F\|_{L^q}+\tilde C\|\nabla\omega\|_{L^q}\\
\leq&\tilde C(\|\rho\dot{u}\|_{L^q}+1)\|\nabla\rho\|_{L^q} +\tilde C\|\rho\dot{u}\|_{L^q} .
\ea
\end{equation}

Then it follows from Lemma  \ref{qo01},  \eqref{433}, and \eqref{437} that
\begin{equation}
\begin{split}
\|\nabla u\|_{L^\infty}\leq& \tilde C(\|\mathrm{div}u\|_{L^\infty}+\|\omega\|_{L^\infty})\mathrm{log}(e+\|\nabla^2u\|_{L^q})+\tilde C\|\nabla u\|_{L^2}+\tilde C\\
\leq& \tilde C(1+\|\rho\dot{u}\|_{L^q})\mathrm{log}(e+\|\nabla\rho\|_{L^q}+\|\rho\dot{u} \|_{L^q})+\tilde C .  \label{438}
\end{split}
\end{equation}
    Noticing that
\begin{equation}
2\mu\|\nabla\rho\|_{L^q}\leq\|\Psi\|_{L^q}\leq \tilde C\|\nabla\rho\|_{L^q},   \label{439}
\end{equation} substituting \eqref{438} into \eqref{436}, and using \eqref{336}, one gets
\begin{equation*}
\frac{d}{dt}\mathrm{log}(e+\|\Psi\|_{L^q})\leq \tilde C(1+\|\rho\dot{u}\|_{L^q})\mathrm{log}(e+\|\Psi\|_{L^q})+\tilde C\|\rho\dot{u}\|^{1+1/q}_{L^q},
\end{equation*}
which together with Gronwall's inequality,  \eqref{336},  \eqref{430}, and \eqref{439} yields that
\begin{equation}
\sup_{0\leq t\leq T}\|\nabla\rho\|_{L^q}\leq \tilde C.  \label{440}
\end{equation}
Combining \eqref{437},   \eqref{440},  and \eqref{430} gives
\begin{equation}
\int_0^T\left(\|\nabla^2u\|_{L^q}^{1+1/q}+t\|\nabla^2u\|_{L^q}^2\right)dt\leq \tilde C. \label{442}
\end{equation}

Finally,  it follows from \eqref{41}, \eqref{402},  \eqref{qe98}, \eqref{3103},  \eqref{410}, and \eqref{442} that
\begin{equation}\label{b442}
\begin{split}
\int_0^Tt\|u_t\|_{H^1}^2dt\leq&\tilde C\int_0^Tt\left(\|\dot u\|^2_{H^1}+\|u\cdot \nabla u\|_{H^1}^2\right)dt\\ \leq& \tilde C\int_0^Tt\left(1+\|\nabla \dot u\|_{L^2}^2+\|\nabla u\|^4_{L^4}+\|u\|_{H^1}^2\|\nabla u\|_{L^q}^2\right)dt\\
\leq&  \tilde C.
\end{split}
\end{equation}

We obtain from \eqref{3103},  \eqref{336}, and \eqref{440} that
\begin{equation*}
\begin{split}
\|\nabla^2u\|_{L^2}\leq& \tilde C(\|\nabla \mathrm{div}u\|_{L^2}+\|\nabla\omega\|_{L^2})\\
\leq&\tilde C(\|\nabla(2\mu+\lambda)\mathrm{div}u)\|_{L^2}
+\|\mathrm{div}u\|_{L^{2q/(q-2)}}\|\nabla\rho\|_{L^q}+\|\nabla\omega\|_{L^2})\\
\leq&\tilde C\|\nabla F\|_{L^2}+\tilde C\|\nabla\omega\|_{L^2}+\tilde C+\tilde C\|\nabla u\|_{L^2}^{(q-2)/q}\|\nabla^2 u\|_{L^2}^{2/q}\\
\leq&\frac{1}{2}\|\nabla^2 u\|_{L^2}+\tilde C\|\rho\dot{u}\|_{L^2}+\tilde C,
\end{split}
\end{equation*}
which together with \eqref{410} gives
\begin{equation*}
\sup_{0\leq t\leq T}t\|\nabla^2u\|^2_{L^2}\leq \tilde C.
\end{equation*}
Combining this, \eqref{440}--\eqref{b442},   and \eqref{3103} yields \eqref{434} and finishes   the proof of Proposition \ref{qou10}.
\end{proof}

\section{\label{sec4}Proofs of Theorems \ref{thmq1} and \ref{thmq2}}
After all preparation done, the proof of Theorem \ref{thmq1} is in the same way as that of \cite[Theorem 1.1]{huang2016existence} after some slight modifications.

We first state the global existence of strong solution $(\rho,  u)$ provided that \eqref{17} holds and that $(\rho_0,  m_0)$ satisfies \eqref{21} whose proof is similar to that of \cite[Proposition 5.1]{huang2016existence} after some slight modifications.
\begin{proposition} \label{thmq3}
Assume that \eqref{17} holds and that $(\rho_0,  m_0)$ satisfies \eqref{21}.  Then there exists a unique strong solution $(\rho,  u)$ to \eqref{11}-\eqref{15} in $\Omega\times(0,  \infty)$ satisfying \eqref{22} for any $T\in(0, \infty)$.  In addition,  for any $q > 2$,  $(\rho,  u)$ satisfies \eqref{434} with some positive constant $C$ depending only on $\Omega,$ $T$,  $q$,  $\mu$,  $\gamma$,  $\beta$,  $\|u_0\|_{H^1}$,  and $\|\rho_0\|_{W^{1, q}}$.
\end{proposition}
\subsection{Proof of Theorem \ref{thmq1}}
Let $(\rho_0,  m_0)$ satisfying \eqref{18} be the initial data as described in Theorem \ref{thmq1}.  We construct a sequence of $C^{\infty}$ initial value $(\rho_0^{\delta}, u_0^{\delta})$,  and $u_0^{\delta}$ should satisfy slip boundary condition.   Using the standard approximation theory to $( {\rho}_0^\delta,u_0)$(see \cite{2010Partial} for example), we can find
a sequence of $C^\infty$ functions $(\tilde{\rho}_0^\delta,\tilde{u}_0^\delta)$ satisfying
\begin{equation}\label{51}
\lim_{\delta\rightarrow 0}\|\tilde{\rho}_0^\delta-\rho_0\|_{W^{1,q}}+\|\tilde{u}_0^\delta-u_0\|_{H^1}=0.
\end{equation}
However, $\tilde{u}_0^\delta$ may violate the slip boundary condition,we must go further.

Let $u^\delta_0$ be the unique smooth solution of elliptic system
\begin{equation}\label{52}
\begin{cases}
  \Delta u^{\delta}=\Delta\tilde{u}^\delta_0 & \mbox{ in } \Omega, \\
   u^{\delta}\cdot n=0, \ \mathrm{curl}u^{\delta}=0& \mbox{ on }\partial\Omega. \\
       \end{cases}
\end{equation} We define $ {\rho}_0^\delta=\tilde\rho^\delta_0+\delta$,
and set $m^\delta_0=\rho^\delta_0 u^\delta_0.$   Then, it is easy to check that
\begin{equation*}
\lim_{\delta\rightarrow 0}(\|\rho^\delta_0-\rho_0\|_{W^{1, q}}+\|u^\delta_0-u_0\|_{H^1})=0.
\end{equation*}

After applying Proposition \ref{thmq3}, we can construct a unique global strong solution $(\rho^{\delta}, u^{\delta})$ with initial value $(\rho_0^{\delta}, u_0^{\delta})$.
Such solution satisfy \eqref{434} for any $T>0$ and for some $C$ independent with $\delta$.

Letting $\delta\rightarrow 0$, standard compactness assertions (see\cite{perepelitsa2006global,vaigant1995,zhang2015}) make sure
the problem \eqref{11}-\eqref{15} has a global strong solution $(\rho, u)$ satisfying the properties listed in Theorem \ref{thmq1}.
Note that the uniqueness can be obtained via similar method in Germain \cite{germain2011weak}.

\subsection{Proof of Theorem \ref{thmq2}}
With all a priori estimations and Theorem \ref{thmq1} done, to deduce
the existence of weak solution, we may follow the proof (especially the compactness assertions) in \cite{huang2016existence}.
We just sketch the proof here. For more details see \cite{huang2016existence}.
\begin{proof}
Let $(\rho_0, u_0)$ be the initial data as in Theorem \ref{thmq2},  we construct an approximation initial value
$(\rho^\delta_0, u^\delta_0)$ in the same manner as \eqref{51} and \eqref{52}, but this time we have for any $p\geq 1$,
\begin{equation*}
\lim_{\delta\rightarrow 0}(\|\rho^\delta_0-\rho_0\|_{L^{p}}+\|u^\delta_0-u_0\|_{H^1})=0.
\end{equation*}
Moreover,
\begin{equation*}
\rho^\delta_0\rightarrow\rho_0\  \mathrm{in}\ W^*\ \mathrm{topology}\ \mathrm{of}\ L^\infty\ \mathrm{as}\ \delta\rightarrow 0.
\end{equation*}
We apply Proposition \ref{thmq3} to deduce the existence of unique global strong solution $(\rho^\delta, u^\delta)$ of problem \eqref{11}-\eqref{15} with initial value $(\rho^\delta_0, u^\delta_0)$ which satisfy \eqref{3103} \eqref{410} \eqref{411}
and \eqref{430} for any $T>0$. The constants involved are all independent of $\delta$.

From \eqref{3103} and \eqref{430}, we argue that
\begin{equation*}
\sup_{0\leq t\leq T}\|u^\delta\|_{H^1}+\int_0^T t\|u^\delta_t\|_{L^2}^2dt\leq C.
\end{equation*}
We apply Aubin-Lions Lemma to get a subsequence that
\begin{equation*}
\begin{split}
&u^\delta\rightharpoonup u \mbox{ weakly  * } \mathrm{in}\ L^\infty(0, T;H^1), \\
&u^\delta\rightarrow u\  \mathrm{in}\ C([\tau, T]; L^p),
\end{split}
\end{equation*}
for any $\tau\in (0, T)$ and $p\geq 1.$

Let $F^\delta=\big(2\mu+(\rho^\delta)^\beta\big)\mathrm{div} u^\delta+P(\rho^\delta)$ be the effect viscous flux with respect to approximation solution $(\rho^\delta, u^\delta)$.
Up on applying \eqref{3103} and \eqref{430} we have
\begin{equation*}
\int_0^T\left(\big \|F^\delta\|_{L^\infty}^{ {4}/{3}} +\|\omega^\delta\|_{H^1}^2 +\|F^\delta\|_{H^1}^{2}
+t\|\omega^\delta_t\|_{L^2}^2+t\|F^\delta_t\|_{L^2}^{2}\right) dt\leq C.
\end{equation*}

Once again, Aubin-Lions Lemma yields
\begin{equation}
\begin{split}\label{59}
&F^\delta\rightharpoonup F \mbox{ weakly  * }  \ \mathrm{ in }\ L^{\frac{4}{3}}(0, T;L^{\infty}),\\
&\omega^\delta\rightarrow \omega, \ F^\delta\rightarrow F\ \mathrm{in}\ L^{2}(\tau, T;L^{p}),
\end{split}
\end{equation}
for any  $\tau\in (0, T)$  and $p\geq 1. $

Now we are in the same position as that of \cite[(5.13)]{huang2016existence}, and we may follow the assertions in \cite{huang2016existence} to deduce the strong convergence of $\rho^\delta$, say
\begin{equation*}
\rho^\delta\rightarrow \rho\ \mathrm{ in }\ C([0,  T];L^p(\Omega)),
\end{equation*}
for any $p\geq 1$. Combining this with \eqref{59}, we argue that
\begin{equation*}
F^\delta\rightarrow F\ \mathrm{in}\ L^2((\tau,  T)\times\Omega)
\end{equation*}
for any $\tau\in (0, T).$  And \eqref{111} follows directly.
We conclude that $(\rho, u)$ is exactly the desired weak solution in Theorem \ref{thmq2} which  finishes our proof.
\end{proof}

\section*{Acknowledgements} The authors would like to thank Prof. Guocai Cai for his valuable discussions.  The research  is
partially supported by the National Center for Mathematics and Interdisciplinary Sciences, CAS,
NSFC Grant Nos.  11688101,  11525106, and 12071200,  and Double-Thousand Plan of Jiangxi Province (No. jxsq2019101008).

\end{document}